\documentclass[12pt,a4paper,oneside]{article}
\usepackage[left=2cm, right=2cm, top=1.3cm, bottom=2cm]{geometry}
\usepackage[utf8]{inputenc}
\usepackage[russian,english]{babel}
\usepackage{amsthm}
\usepackage{amsmath}
\usepackage{amssymb}
\usepackage{xcolor}
\usepackage[colorlinks=true, linkcolor=blue]{hyperref}
\usepackage[english]{cleveref}
\usepackage{caption}
\usepackage{subcaption}
\input{insbox}

\usepackage{tikz}

\usetikzlibrary{decorations.pathreplacing}
\usetikzlibrary{arrows,calc,patterns,positioning,shapes}
\usetikzlibrary{decorations.pathmorphing}
\usetikzlibrary{cd}

\tikzset{
  modal/.style={>=stealth',shorten >=1pt,shorten <=1pt,auto,
  node distance=1.5cm,semithick},
  world/.style={circle,draw,minimum size=0.8cm,fill=gray!15},
  branch/.style={rectangle,draw,inner sep=0.75mm},
  leaf/.style={circle,draw,inner sep=0.5mm},
  l-leaf/.style={left = of #1, xshift=10mm, yshift=2mm},
  r-leaf/.style={right = of #1, xshift=-7mm, yshift=2mm},
  br/.style={->,dashed},
  l/.style={->,out=270,in=270},
  l-branch/.style={above left = of #1, yshift = 0.5cm},
  r-branch/.style={above right = of #1, yshift = 0.5cm},
  point/.style={circle,draw,inner sep=0.5mm},
  plane/.style={trapezium,draw,
    minimum width=4cm, minimum height=1cm,
    trapezium left angle=50, trapezium right angle=130,
    trapezium stretches=false,
    outer sep=1mm},
  flat-plane/.style={trapezium,draw,
    minimum width=4cm, minimum height=1cm,
    trapezium left angle=40, trapezium right angle=140,
    trapezium stretches=false,
    outer sep=1mm},
  l-lbl/.style={label=195:{#1}},
  root-lbl/.style={label={[yshift=1mm]270:#1}},
  l-plane/.style={above left = of #1, xshift = 1.7cm,yshift=-4mm},
  r-plane/.style={above right = of #1, xshift = 0.5cm,yshift=-4mm},
  t-lbl/.style={label={[xshift=-1mm,yshift=1mm]below:{\footnotesize #1}}},
  w1-lbl/.style={label={[xshift=-2mm,yshift=-1mm]above:{\footnotesize #1}}},
  w2-lbl/.style={label={[xshift=1mm,yshift=1mm]below:{\footnotesize #1}}},
  w3-lbl/.style={label={[xshift=4mm,yshift=1mm]below:{\footnotesize #1}}},
  symbol/.style={
    draw=none,
    every to/.append style={
      edge node={node [sloped, allow upside down, auto=false]{$#1$}}}
  },
}

\usepackage{mathtools}
\newcommand{\defeq}{\vcentcolon=}

\newcommand{\Dmd}{\Diamond}
\newcommand{\impl}{\rightarrow}
\renewcommand{\iff}{\leftrightarrow}
\newcommand{\Iff}{\Leftrightarrow}

\usepackage{thmtools}
\usepackage{thm-restate}

\newcommand{\mkTheorem}[4]{
  \declaretheorem[#4,name=#2,refname={#2,#3},Refname={#2,#3}]{#1}
}

\mkTheorem{definition}{Definition}{Definitions}{style=definition,numberwithin=section}
\mkTheorem{theorem}{Theorem}{Theorems}{numberwithin=section}
\mkTheorem{proposition}{Proposition}{Propositions}{numberwithin=section}
\mkTheorem{lemma}{Lemma}{Lemmas}{numberwithin=section}
\mkTheorem{corollary}{Corollary}{Corollaries}{numberwithin=section}
\crefname{equation}{}{}
\Crefname{equation}{}{}
\crefname{section}{Section}{Section}
\numberwithin{equation}{section}

\usepackage{titling}
\newcommand{\subtitle}[1]{%
  \posttitle{%
    \par\end{center}
    \begin{center}\large#1\end{center}
    \vskip0.5em}%
}

\DeclareFontFamily{U}{mathb}{\hyphenchar\font45}
\DeclareFontShape{U}{mathb}{m}{n}{
  <-6> mathb5 <6-7> mathb6 <7-8> mathb7
  <8-9> mathb8 <9-10> mathb9
  <10-12> mathb10 <12-> mathb12
}{}
\DeclareSymbolFont{mathb}{U}{mathb}{m}{n}
\DeclareMathSymbol{\llcurly}{\mathrel}{mathb}{"CE}
\DeclareMathSymbol{\ggcurly}{\mathrel}{mathb}{"CF}

\usepackage{graphics}

\usepackage{stackengine}

\makeatletter
\newcommand{\labitem}[2]{%
\def\@itemlabel{\textbf{#1}}
\item
\def\@currentlabel{#1}\label{#2}}
\makeatother

\usepackage{multicol}
\usepackage{wrapfig}
\usepackage{graphicx}
\usepackage{centernot}

\usepackage{cases}

\author{Lev Dvorkin\thanks{The author is a winner of the ``Junior Leader''
    competition grant held by the ``BASIS'' Foundation for the Development of
    Theoretical Physics and Mathematics.}}
\title{Finite model property of pretransitive \mbox{analogues of (w)K4 and GL}}
\date{\today}

\newcommand{\calM}{\mathcal M}
\newcommand{\calF}{\mathcal F}

\newcommand{\PA}{\mathrm{PA}}

\newcommand{\defw}{\textit}

\newcommand{\E}[1]{\exists #1\,}
\newcommand{\A}[1]{\forall #1\,}

\begin{document}
  \maketitle \begin{abstract}
  A normal modal logic is pretransitive, if the modality corresponding to the transitive closure of
  an accessibility relation is expressible in it. In the present work we establish the finite model
  property for pretransitive generalizations of $\mathrm{K4}$, $\mathrm{wK4}$, $\mathrm{GL}$, and
  their extensions by canonical subframe-hereditary formulas.
\end{abstract}

\section{Introduction}
  Let $\Lambda$ be a normal modal logic. A formula $\chi$ with a single variable $p$
  \defw{expresses} the transitive closure in $\Lambda$, if, for any $\Lambda$-model
  $\calM = (W, R, \vartheta)$ and world $w \in W$, $\calM, w \vDash \chi$ iff there is a world $v$
  accessible from $w$ by the transitive closure of $R$ in which $p$ is true. Expressibility of the
  reflexive transitive closure is defined in a similar way. It is easy to see that if, for some
  $n \geq 1$, the formula
  \begin{gather*}
    \mathrm{Trans}_n \defeq \Dmd^{n+1}p \impl \bigvee_{k=1}^n \Dmd^k p\quad \left(\mathrm{wTrans}_n
    \defeq \Dmd^{n+1}p \impl \bigvee_{k=0}^n \Dmd^k p\right),
  \end{gather*}
  is derivable in $\Lambda$, then $\bigvee_{k=1}^n \Dmd^k p$ ($\bigvee_{k=0}^n \Dmd^k p$) expresses
  the (reflexive) transitive closure in it. We call a logic \defw{(weakly) $n$-transitive} if it
  contains $\mathrm{(w)Trans}_n$. It is obvious that any $n$-transitive logic is weakly
  $n$-transitive. In fact, the following conditions on a logic $\Lambda$ are equivalent
  (\cref{pr:pretrans}):
  \begin{itemize}
    \item transitive closure is expressible in $\Lambda$;
    \item reflexive transitive closure is expressible in $\Lambda$;
    \item $\Lambda$ is $n$-transitive for some $n \geq 1$;
    \item $\Lambda$ is weakly $n'$-transitive for some $n' \geq 1$.
  \end{itemize}
  If any of these conditions holds, $\Lambda$ is called \defw{pretransitive}.

  The simplest examples of pretransitive logics are
  \begin{gather*}
    \mathrm{K4} \defeq \mathrm{K} + \Dmd^2 p \impl \Dmd p \quad\text{and}\quad \mathrm{wK4} \defeq
    \mathrm{K} + \Dmd^2 p \impl \Dmd p \vee p
  \end{gather*}
  {}--- the least $1$-transitive and weakly $1$-transitive logics. It is known that both logics have
  the finite model property \cite{Lemm66,BEG09}. Moreover, all extensions of these logics by
  \defw{subframe-hereditary} formulas (if a formula is valid in a frame $\calF$, then it is valid in
  all subframes of $\calF$) have the finite model property~\cite{Fine85,BGJ11}. In contrast, the
  finite model property of $\mathrm{K} + \mathrm{Trans}_n$ and $\mathrm{K} + \mathrm{wTrans}_n$ for
  $n \geq 2$ is an open problem~\cite[Problem 11.2]{ChZa97}.

  Another examples of pretransitive logics are
  \begin{gather*}
    \mathrm{K4}^m_n \defeq \mathrm{K} + \Dmd^n p \impl \Dmd^m p \quad\text{and}\quad
    \mathrm{wK4}^m_n \defeq \mathrm{K} + \Dmd^n p \impl \Dmd^m p \vee p
  \end{gather*}
  for $m < n$. If $m = 0$, these logics are degenerated and obviously have the finite model
  property. The finite model property of $\mathrm{K4}^1_n$ was established by Gabbay~\cite{Gab72},
  the same property of $\mathrm{wK4}^1_n$ was established in the recent work of Kudinov and
  Shapirovsky \cite{KudShap25}. Moreover, in the work~\cite{KudShap25}, it was proven that all
  extensions of these logics by canonical subframe-hereditary formulas have the finite model
  property. For $n > m \geq 2$, it is unknown whether $\mathrm{K4}^m_n$ and $\mathrm{wK4}^m_n$ have
  the finite model property~\cite[Problem 11.2]{ChZa97}.

  The situation is better for the logics of finite height and, in particular, for pretransitive
  generalizations of $\mathrm{S5}$. The finite model property of extensions of
  $\mathrm{K} + \mathrm{wTrans}_n$, $n \geq 1$, and $\mathrm{K4}^m_n$, $n > m \geq 1$, by the
  symmetry axiom for the transitive closure was established in~\cite{Jan94} and \cite{KudShap11}
  respectively. In~\cite{KudShap17}, the finite model property of extensions of these logics by the
  bounded height axioms was obtained.

  Besides pretransitive analogues of $\mathrm{K4}$ and $\mathrm{wK4}$, the following generalizations
  of G\"odel--L\"ob logic were considered in the previous works:
  \begin{gather*}
    \mathrm{GL}_n \defeq \mathrm{K} + \Box(\Box^{n-1} p \impl p) \impl \Box p = \mathrm{K} + \Dmd p
    \impl \Dmd(p \wedge \neg\Dmd^{n-1}p),\quad n \geq 2.
  \end{gather*}
  Sacchetti~\cite{Sacc01} proved that these logics are complete with respect to the class of all
  finite conversely well-founded $\mathrm{K4}^1_n$-frames, have Craig interpolation and fixed point
  properties. Kurahashi~\cite{Kur18} constructed $\Sigma_2$ numerations for $\PA$, for which these
  logics are provability logics.

  Another pretransitive analogue of $\mathrm{GL}$
  \begin{gather}
    \label{eq:gl-2}
    \mathrm{K} + \Box\bigl(\Box(p \vee \Box p) \impl p\bigr) \impl \Box p = \mathrm{K} + \Dmd p
    \impl \Dmd\bigl(p \wedge \neg\Dmd(p \wedge \Dmd p)\bigr)
  \end{gather}
  appeared in~\cite{Dvo24} as a fragment of the provability logic of a non-r.e.\ extension of $\PA$
  with respect to $\PA$. The properties of this logic have not been studied yet.

  Let us denote the family of logics $\mathrm{K4}^1_n$, $\mathrm{wK4}^1_n$, and $\mathrm{GL}_n$ for
  $n \geq 2$ by $C_+$. We also denote the family of logics $\mathrm{K} + \mathrm{Trans}_n$,
  $\mathrm{K} + \mathrm{wTrans}_n$ for $n \geq 2$ and $\mathrm{K4}^m_n$, $\mathrm{wK4}^m_n$ for
  $n > m \geq 2$ by $C_?$. As it was mentioned above, the finite model property for logics from
  $C_+$ is known, but for logics from $C_?$ it is an open problem. There are several properties of
  the logics from $C_+$ that make them easier to investigate. Firstly, the logics from $C_+$ are
  subframe-hereditary. Secondly, for every $\Lambda \in C_+$ and Kripke frame $\calF = (W, R)$ there
  is a \defw{$\Lambda$-closure} of $R$ (the smallest relation $R' \supseteq R$ such that $(W, R')$
  is a $\Lambda$-frame). For example, $\mathrm{K4}$-closure is the transitive closure and
  $\mathrm{K4}^1_n$-closure of $R$ is $\bigcup_{k=0}^\infty R^{(n-1)k+1}$. The logics from $C_?$
  satisfy neither of these properties.

  However, there are some logics with the first property, but without the second. The following
  logics are among them:
  \begin{align*}
    \mathrm{K4}_{\Dmd^2\alpha} &\defeq \mathrm{K} + \Dmd^2 \alpha(p) \impl \Dmd p,\\
    \mathrm{wK4}_{\Dmd^2\alpha} &\defeq \mathrm{K} + \Dmd^2 \alpha(p) \impl \Dmd p \vee p,\\
    \mathrm{GL}_{\Dmd^2\alpha} &\defeq \mathrm{K} + \Dmd p \impl \Dmd(p \wedge \neg\Dmd\alpha(p)),
  \end{align*}
  where $\alpha$ is a strictly positive formula with exactly one variable $p$. Denote by $C$ the
  family of these logics. Notice that $C \supset C_+$:
  $\mathrm{(w)K4}_{\Dmd^n p} = \mathrm{(w)K4}^1_n$ and $\mathrm{GL}_{\Dmd^n p} = \mathrm{GL}_n$ for
  $n \geq 2$. Also, $\mathrm{GL}_{\Dmd^2(p \wedge \Dmd p)} \in C$ is exactly the logic presented
  in~\cref{eq:gl-2}. We will show that all logics from $C$ and their extensions by canonical
  subframe-hereditary formulas (for $\mathrm{GL}_{\Dmd^2\alpha}$, inherited by more general
  substructures) have the finite model property.

  The paper is organized as follows:

  In \cref{s:basics}, we provide neccessary definitions and well-known facts from modal logic.

  In \cref{s:class}, we define the logics under consideration, prove their basic properties and
  state the main results.

  In \cref{s:grid}, we obtain auxiliary combinatorial results about Kripke frames of the logics
  under consideration.

  In \cref{s:fmp}, we prove the finite model property of the pretransitive analogues of
  $\mathrm{(w)K4}$ and $\mathrm{GL}$.

  In \cref{s:gl-equiv}, we extend the well-known fact that $\mathrm{GL}$ is the closure of
  $\mathrm{wK4}$ under the L\"ob's rule to the pretransitive case.

  In \cref{s:problems}, we state some open problems for the future work.
 \section{Basic concepts}
  The most of the definitions here are standard, see, e.g., \cite{ChZa97} or \cite{BdRV01}.

  \label{s:basics}
  \subsection{Modal formulas and logics}
    The set of \defw{modal formulas} $\mathrm{Fm}$ is built from a countable set of
    \defw{propositional variables} $\mathrm{PVar} = \{p_0, p_1, \dots\}$ and $\bot$ using
    implication $\impl$ and an unary modal connective $\Dmd$. Other connectives are defined as
    abbreviations, in particular, $\neg\varphi \defeq \varphi \impl \bot$, $\top \defeq \neg\bot$,
    and $\Box\varphi \defeq \neg\Dmd\neg\varphi$. Also we denote $\Dmd^0 \varphi \defeq \varphi$,
    $\Dmd^{k+1} \varphi \defeq \Dmd\Dmd^k \varphi$, $\Box^k \varphi \defeq \neg\Dmd^k\neg\varphi$
    for $k \in \omega$. Graphical equality of formulas is denoted by the symbol $\equiv$.

    The set of all formulas with a single variable $p \in \mathrm{PVar}$ is denoted by
    $\mathrm{Fm}(p)$. For $\alpha \in \mathrm{Fm}(p)$ and $\varphi \in \mathrm{Fm}$, the formula
    obtained by substituting $\varphi$ for $p$ in $\alpha$ is denoted by $\alpha(\varphi)$. In
    particular, $\alpha(p) \equiv \alpha$.

    A \defw{(propositional normal modal) logic} is a set of formulas $\Lambda \subseteq \mathrm{Fm}$
    which contains all classical tautologies, the axioms $\Dmd\bot \impl \bot$ and
    $\Dmd(p \vee q) \impl \Dmd p \vee \Dmd q$ and is closed under the rules of modus ponens,
    substitution, and \defw{monotonicity} $\frac{\varphi \impl \psi}{\Dmd\varphi \impl \Dmd\psi}$.
    The smallest logic is denoted by $\mathrm{K}$. The largest logic is the set of all formulas. All
    logics except $\mathrm{Fm}$ are called \defw{consistent}. For a logic $\Lambda$ and a formula
    $\varphi$, we denote by $\Lambda + \varphi$ the smallest logic containing $\Lambda$ and
    $\varphi$. We use standard notation for the logics
    \begin{align*}
      \mathrm{K4} &\defeq \mathrm{K} + \Dmd^2 p \impl \Dmd p,\\
      \mathrm{wK4} &\defeq \mathrm{K} + \Dmd^2 p \impl \Dmd p \vee p,\\
      \mathrm{GL} &\defeq \mathrm{K} + \Dmd p \impl \Dmd(p \wedge \neg\Dmd p).
    \end{align*}
    We say that a formula $\varphi$ \defw{is derivable} in $\Lambda$ and write
    $\Lambda \vdash \varphi$, if $\varphi \in \Lambda$.

  \subsection{Kripke frames and models}
    A \defw{Kripke frame} is a pair $\calF = (W, R)$, where $W$ is a nonempty set, and
    $R \subseteq W \times W$ is a relation on $W$. The elements of $W$ are called \defw{possible
    worlds} and $R$ is called an \defw{accessibility relation}. $\calF$ is \emph{finite}, if $W$ is
    finite. It is \emph{transitive}, if $R$ is transitive. $\calF$ is \emph{conversely
    well-founded}, if there is no infinite sequence of worlds $w_0, w_1, \dots$ such that
    $w_k \mathrel R w_{k+1}$ for all $k \in \omega$.

    For a world $w \in W$, we denote by $R(w)$ the set $\{v \in W \mid w \mathrel R v\}$. Powers of
    a relation are defined in the standard way: $R^0(w) \defeq \{w\}$,
    $R^{k+1}(w) \defeq \bigcup_{v \in R(w)}R^k(v)$, $k \in \omega$. Relations
    $R^+ \defeq \bigcup_{k > 0}R^k$ and $R^* \defeq \bigcup_{k \geq 0}R^k$ are called
    \defw{transitive} and \defw{reflexive transitive} closure of $R$ respectively.

    A \defw{valuation} on the frame $\calF$ is a mapping
    $\vartheta : \mathrm{PVar} \impl \mathcal{P}(W)$, where $\mathcal{P}(W)$ is the set of all
    subsets of $W$. A \defw{Kripke model} on $\calF$ is a pair $\calM = (\calF, \vartheta)$, where
    $\vartheta$ is a valuation on $\calF$. The valuation $\vartheta$ is extended to $\mathrm{Fm}$ by
    the following rules:
    \begin{gather*}
      \vartheta(\bot) \defeq \emptyset,\quad \vartheta(\varphi \impl \psi) \defeq W \setminus
      \vartheta(\varphi) \cup \vartheta(\psi),\\
      \vartheta(\Dmd\varphi) \defeq \{w \in W \mid R(w) \cap \vartheta(\varphi) \neq \emptyset \}.
    \end{gather*}
    It is clear that
    $\vartheta(\Dmd^k \varphi) = \{w \in W \mid R^k(w) \cap \vartheta(\varphi) \neq \emptyset\}$. If
    $w \in \vartheta(\varphi)$, we say that $\varphi$ \defw{is true at $w$ in $\calM$} and denote
    this by $\calM, w \vDash \varphi$. A formula $\varphi$ is \defw{true in a model} $\calM$, if
    $\vartheta(\varphi) = W$, \defw{valid in a frame} $\calF$, if it is true in all models on
    $\calF$, \defw{valid in a class of frames} $\mathcal{C}$, if it is valid in all frames from
    $\mathcal{C}$. We denote this by $\calM \vDash \varphi$, $\calF \vDash \varphi$, and
    $\mathcal{C} \vDash \varphi$ respectively. For $\Gamma \subseteq \mathrm{Fm}$, we say that
    \emph{$\Gamma$ is true at $w$ in $\calM$}, if all $\varphi \in \Gamma$ are true at $w$ and
    denote this by $\calM, w \vDash \Gamma$. $\calM \vDash \Gamma$, $\calF \vDash \Gamma$, and
    $\mathcal{C} \vDash \Gamma$ are defined in a similar way.

    A Kripke model (frame) is called a \defw{$\Lambda$-model ($\Lambda$-frame)} if $\Lambda$ is true
    (valid) in it. The class of all $\Lambda$-frames is denoted by $\mathrm{Frames}(\Lambda)$. The
    set $\mathrm{Log}(\mathcal{C})$ of all formulas that are valid in a class of frames
    $\mathcal{C}$ is called the \defw{logic of}~$\mathcal{C}$. It is easy to show that, if
    $\mathcal{C} \neq \emptyset$, then $\mathrm{Log}(\mathcal{C})$ is a consistent logic. A logic
    $\Lambda$ is \defw{Kripke-complete}, if it is the logic of some class of frames. $\Lambda$ has
    the \defw{finite model property}, if $\Lambda$ is the logic of a class of finite frames.
  \subsection{Generated and selective subframes and submodels}
    For a Kripke model $\calM_0 = (W_0, R_0, \vartheta_0)$ and $W \subseteq W_0$, we denote
    \begin{gather*}
      R_0|_W \defeq R_0 \cap (W \times W),\quad \vartheta_0|_W(p) \defeq \vartheta_0(p) \cap W\text{
      for all } p \in \mathrm{PVar}.
    \end{gather*}
    \vspace{-3.5ex}
    \begin{definition}
      Let $\calF = (W, R)$ and $\calF_0 = (W_0, R_0)$ be Kripke frames. $\calF$ is
      \begin{itemize}
        \item a \emph{weak subframe} of $\calF_0$ if $W \subseteq W_0$ and $R \subseteq R_0$;
        \item a \emph{subframe} of $\calF_0$ if $W \subseteq W_0$ and $R = R_0|_W$;
        \item a \emph{generated subframe} of $\calF_0$ if it is a subframe of $\calF_0$ and
              $R_0(w) \subseteq W$ for all $w \in W$.
      \end{itemize}
      $\calM = (\calF, \vartheta)$ is a \emph{submodel} (\emph{weak submodel, generated submodel})
      of $\calM_0 = (\calF_0, \vartheta_0)$ if $\vartheta = \vartheta_0|_W$ and $\calF$ is a
      subframe (weak subframe, generated subframe) of $\calF_0$.
    \end{definition}
    It is easy to see that the truth of modal formulas is preserved in generated submodels: if
    $\calM$ is a generated submodel of $\calM_0$, then
    $\vartheta(\varphi) = \vartheta_0(\varphi) \cap W$ for all $\varphi \in \mathrm{Fm}$. It follows
    that validity is preserved in generated subframes, that is,
    $\mathrm{Log}(\calF_0) \subseteq \mathrm{Log}(\calF)$ whenever $\calF$ is a generated subframe
    of $\calF_0$. However, the truth in not preserved in submodels in general. For this purpose,
    selective submodels are considered.

    For a formula $\zeta$, denote by $\Psi^\zeta$ the set of all formulas $\psi$ such that
    $\Dmd\psi$ is a subformula of $\zeta$.
    \begin{definition}
      A weak submodel $\calM = (W, R, \vartheta)$ of the Kripke model
      $\calM_0 = (W_0, R_0, \vartheta_0)$ is \emph{$\zeta$-selective}, if
      \begin{gather*}
        \A{ w \in W} \A{\psi \in \Psi^\zeta} \bigl(R_0(w) \cap \vartheta_0(\psi) \neq \emptyset
        \Rightarrow R(w) \cap \vartheta_0(\psi) \neq \emptyset\bigr).
      \end{gather*}
    \end{definition}
    These kind of submodels are also called selective filtrations~\cite{ShapSheht03}. The following
    standard lemma claims that the truth of $\zeta$ and all its subformulas is preserved in a
    $\zeta$-selective submodel~\cite[Lemma 6]{ShapSheht03}:
    \begin{lemma}
      \label{l:filtration} If $\calM$ is a $\zeta$-selective weak submodel of $\calM_0$, then
      $\vartheta(\varphi) = \vartheta_0(\varphi) \cap W$ for each subformula $\varphi$ of $\zeta$.
    \end{lemma}
  \subsection{Canonical model}
    Let $\Lambda$ be a modal logic. A set of formulas $\Gamma$ is called
    \defw{$\Lambda$-consistent}, if there is no finite subset $\Gamma' \subseteq \Gamma$ such that
    $\Lambda \vdash \bigwedge\Gamma' \impl \bot$. For consistent $\Lambda$, there is a model
    $\calM_\Lambda = (W_\Lambda, R_\Lambda, \vartheta_\Lambda)$, which is called the \emph{canonical
    model of $\Lambda$}, satisfying the following properties:
    \begin{enumerate}
      \item \label{c:model} $\Lambda$ is true in $\calM_\Lambda$.
      \item \label{c:lind} \defw{Lindenbaum's lemma}: for any $\Lambda$-consistent set
            $\Gamma \subset \mathrm{Fm}$, there is a world $w \in W_\Lambda$ such that
            $\calM_\Lambda, w \vDash \Gamma$.
      \item \label{c:rel} $\calM_\Lambda$ is \defw{tight}: for any two worlds $w, v \in W_\Lambda$,
            \begin{gather*}
              \A{\varphi \in \mathrm{Fm}}(\calM_\Lambda, w \vDash \Box\varphi \Rightarrow
              \calM_\Lambda, v \vDash \varphi) \Rightarrow w \mathrel R_\Lambda v.
            \end{gather*}
      \item \label{c:diff} $\calM_\Lambda$ is \defw{differentiated}: for two distinct worlds
            $w, v \in W_\Lambda$, there is a formula $\varphi$ such that
            $\calM_\Lambda, w \vDash \varphi$ and $\calM_\Lambda, v \nvDash \varphi$.
    \end{enumerate}
    The frame $\calF_\Lambda \defeq (W_\Lambda, R_\Lambda)$ is called the \defw{canonical frame} of
    $\Lambda$. $\Lambda$ is called \defw{canonical} if $\calF_\Lambda \vDash \Lambda$. If $\Lambda$
    is canonical, then, by Lindenbaum's lemma, $\Lambda = \mathrm{Log}(\calF_\Lambda)$, whence
    $\Lambda$ is Kripke-complete.

    Also we need the following standard facts about the canonical models:
    \begin{lemma}
      \label{l:canon-subframe} For any two consistent logics $\Lambda$ and $\Lambda'$, if
      $\Lambda \subseteq \Lambda'$, then $\calM_{\Lambda'}$ is a generated submodel of
      $\calM_\Lambda$.
    \end{lemma}
    \begin{corollary}
      \label{c:canon-subframe} If $\Lambda$ is canonical and $\Lambda' \supseteq \Lambda$, then
      $\calF_{\Lambda'} \vDash \Lambda$.
    \end{corollary}
    A formula $\varphi$ is called \defw{canonical}, if the logic $\mathrm{K} + \varphi$ is
    canonical. By \cref{c:canon-subframe}, if $\Lambda \vdash \varphi$ and $\varphi$ is canonical,
    then $\calF_\Lambda \vDash \varphi$. A lot of examples of canonical formulas are given by
    Sahlqvist theorem~\cite{Sahlq75}. In particular, $\mathrm{Trans}_n$ and $\mathrm{wTrans}_n$ are
    Sahlqvist and, therefore, canonical.
    \begin{lemma}[generalized tightness property]
      Let $\Lambda$ be a consistent logic, $w, v \in W_\Lambda$. The following holds:
      \begin{enumerate}
        \item For all $k \in \omega$, if $\A{\varphi \in \mathrm{Fm}}(\calM_\Lambda, w \vDash \Box^k
              \varphi \Rightarrow \calM_\Lambda, v \vDash \varphi)$, then
              $w \mathrel R^k_\Lambda v$.
        \item If, for some $n \geq 1$, $\A{\varphi \in \mathrm{Fm}}(\calM_\Lambda, w \vDash
              \bigwedge_{k=1}^n\Box^k \varphi \Rightarrow \calM_\Lambda, v \vDash \varphi)$, then
              $w \mathrel R^+_\Lambda v$.
      \end{enumerate}
    \end{lemma}
    \begin{proof}
      The first claim is standard (see \cite[Proposition 5.9]{ChZa97}). For the second claim,
      suppose that $w \mathrel R^+_\Lambda v$ does not hold. Then, for all $k \geq 1$,
      $w \mathrel R^k_\Lambda v$ does not hold and, by the first claim, there are
      $\varphi_k \in \mathrm{Fm}$ such that $\calM_\Lambda, w \vDash \Box^k \varphi_k$ and
      $\calM_\Lambda, v \nvDash \varphi_k$. Let $\varphi \defeq \bigvee_{k=1}^n \varphi_k$. It is
      easy to see that $\calM_\Lambda, v \nvDash \varphi$ and
      $\calM_\Lambda, w \vDash \Box^k \varphi$ for all $k \in \{1, \dots, n\}$. Thus, the premise of
      the claim is false.
    \end{proof}
  \subsection{Pretransitive logics}
    For a variable $p \in \mathrm{PVar}$ and a formula $\varphi$, consider the set
    $D_p(\varphi) \subseteq \omega$:
    \begin{gather*}
      D_p(p) \defeq \{0\},\quad D_p(\bot) = D_p(q) \defeq \emptyset \quad\text{for }q \in
      \mathrm{PVar} \setminus \{p\},\\
      D_p(\varphi \impl \psi) \defeq D_p(\varphi) \cup D_p(\psi),\quad D_p(\Dmd\varphi) \defeq \{k +
      1 \mid k \in D_p(\varphi)\}.
    \end{gather*}
    The \defw{modal depth} $d(\varphi)$ of a formula $\varphi$ is defined recursively
    \begin{gather*}
      d(p) = d(\bot) = 0,\quad d(\varphi \impl \psi) = \mathrm{max}\{d(\varphi), d(\psi)\},\quad
      d(\Dmd\varphi) = d(\varphi) + 1.
    \end{gather*}
    It is easy to see that if $k \in D_p(\varphi)$, then $k \leq d(\varphi)$.
    \begin{lemma}
      \label{l:modalized} Let $\calF = (W, R)$ be a Kripke frame, $w \in W$, and
      $\varphi \in \mathrm{Fm}$. If valuations $\vartheta$ and $\vartheta'$ on $\calF$ are such that
      \begin{gather}
        \vartheta'(p) \cap \bigcup_{k \in D_p(\varphi)} R^k(w) = \vartheta(p) \cap \bigcup_{k \in
        D_p(\varphi)} R^k(w) \label{eq:modalized}
      \end{gather}
      for all $p \in \mathrm{PVar}$, then
      $(\calF, \vartheta), w \vDash \varphi \Iff (\calF, \vartheta'), w \vDash \varphi$.
    \end{lemma}
    \begin{proof}
      We proceed by induction on construction of $\varphi$. For $\varphi \equiv \bot$, the statement
      is trivial. If $\varphi \equiv p \in \mathrm{PVar}$, then $D_p(\varphi) = \{0\}$ and,
      by~\cref{eq:modalized}, $w \in \vartheta'(p) \Iff w \in \vartheta(p)$. The induction step for
      $\varphi \equiv \psi \impl \eta$ is straightforward.

      Consider the case $\varphi \equiv \Dmd\psi$. By~\cref{eq:modalized}, for all
      $p \in \mathrm{PVar}$,
      \begin{gather*}
        \vartheta'(p) \cap \bigcup_{k \in D_p(\psi)} R^{k+1}(w) = \vartheta(p) \cap \bigcup_{k \in
        D_p(\psi)} R^{k+1}(w).
      \end{gather*}
      Notice that, for $v \in R(w)$, $R^k(v) \subseteq R^{k+1}(w)$, whence, for all
      $p \in \mathrm{PVar}$,
      \begin{gather*}
        \vartheta'(p) \cap \bigcup_{k \in D_p(\psi)} R^{k}(v) = \vartheta(p) \cap \bigcup_{k \in
        D_p(\psi)} R^{k}(v)
      \end{gather*}
      By the induction hypothesis, $v \in \vartheta(\psi) \Iff v \in \vartheta'(\psi)$. Thus,
      $w \in \vartheta(\Dmd\psi) \Iff w \in \vartheta'(\Dmd\psi)$.
    \end{proof}
    \begin{proposition}
      \label{pr:pretrans} The following conditions on a consistent logic $\Lambda$ are equivalent:
      \begin{enumerate}
        \item Transitive closure is expressible in $\Lambda$, that is, there is
              $\chi \in \mathrm{Fm}(p)$ such that, for every $\Lambda$-model $(W, R, \vartheta)$,
              \begin{gather*}
                \vartheta(\chi) = \{w \in W \mid R^+(w) \cap \vartheta(p) \neq \emptyset \};
              \end{gather*}
        \item Reflexive transitive closure is expressible in $\Lambda$, that is, there is
              $\chi' \in \mathrm{Fm}(p)$ such that, for every $\Lambda$-model $(W, R, \vartheta)$,
              \begin{gather}
                \vartheta(\chi') = \{w \in W \mid R^*(w) \cap \vartheta(p) \neq \emptyset \}
                \label{eq:chi-prime};
              \end{gather}
        \item There is $n \in \omega$ such that $\Lambda$ is $n$-transitive, that is,
              $\Lambda \vdash \Dmd^{n+1}p \impl \bigvee_{k=1}^n \Dmd^k p$;
        \item There is $n' \in \omega$ such that $\Lambda$ is weakly $n'$-transitive, that is,
              $\Lambda \vdash \Dmd^{n'+1}p \impl \bigvee_{k=0}^{n'} \Dmd^k p$.
      \end{enumerate}
    \end{proposition}
    Logics satisfying these equivalent conditions are known under various names: \emph{weakly
    transitive}~\cite{Kra99,KowKra06}, \emph{conically expressive}~\cite{GabSkvoSheht09}, and
    \emph{pretransitive}~\cite{KudShap11}. In the present paper we use the last term. Also, it is
    more common to consider reflexive transitive closure rather than transitive one, so weakly
    $n$-transitive logics are often called simply \emph{$n$-transitive}~\cite{Kra99,KudShap17}. In
    fact, equivalence $2 \Iff 4$ is known (it follows easily from~\cite[Proposition 5]{KowKra06}).
    However, we provide a complete proof of~\cref{pr:pretrans} here:
    \begin{proof}
      Implications $4 \Rightarrow 3$, $3 \Rightarrow 1$, and $1 \Rightarrow 2$ are trivial: it is
      sufficient to set $n \defeq n' + 1$, $\chi \defeq \bigvee_{k=1}^n \Dmd^k p$, and
      $\chi' \defeq \chi \vee p$ respectively.

      Let us check that $2 \Rightarrow 4$. Suppose that $\chi'$ satisfies the condition
      \cref{eq:chi-prime} and, for $n' \defeq d(\chi')$, the implication
      $\Dmd^{n'+1}p \impl \bigvee_{k=0}^{n'} \Dmd^k p$ is not derivable in $\Lambda$. Then, by
      Lindenbaum's lemma, there is a world $w \in W_\Lambda$ such that
      $\calM_\Lambda, w \vDash \Dmd^{n'+1}p$ and $\calM_\Lambda, w \nvDash \Dmd^k p$ for all
      $k \leq n'$.

      Consider the valuation $\vartheta_\emptyset$ on $\calF_\Lambda$ mapping all variables to
      $\emptyset$. By induction on construction of a formula $\varphi \in \mathrm{Fm}$, one can
      easily check that $\vartheta_\emptyset(\varphi) = \vartheta_\Lambda(\varphi_\bot)$, where
      $\varphi_\bot$ is a variable free formula obtained from $\varphi$ by substituting $\bot$ for
      all variables. Since $\Lambda$ is closed under the substitution rule,
      $\calM_\emptyset \defeq (\calF_\Lambda, \vartheta_\emptyset)$ is a $\Lambda$-model.

      Hence, $\calM_\Lambda$ and $\calM_\emptyset$ are $\Lambda$-models,
      $R_\Lambda^+(w) \cap \vartheta_\Lambda(p) \neq \emptyset$ and
      $R_\Lambda^+(w) \cap \vartheta_\emptyset(p) = \emptyset$. Therefore,
      $\calM_\Lambda, w \vDash \chi'$ and $\calM_\emptyset, w \nvDash \chi'$. However,
      \begin{gather*}
        \vartheta_\Lambda(p) \cap \bigcup_{k \in D_p(\chi')} R_\Lambda^k(p) \subseteq
        \vartheta_\Lambda(p) \cap \bigcup_{k=0}^{n'} R_\Lambda^k(p) = \emptyset =
        \vartheta_\emptyset(p) \cap \bigcup_{k \in D_p(\chi')} R_\Lambda^k(p).
      \end{gather*}
      This is impossible by~\cref{l:modalized}.
    \end{proof}
    A Kripke frame $\calF = (W, R)$ is \defw{$n$-transitive}, if $R^+ = \bigcup_{k=1}^n R^k$. Denote
    \begin{gather*}
      \Dmd^{+n}\varphi \defeq \bigvee_{k=1}^n \Dmd^k\varphi,\quad \Box^{+n}\varphi \defeq
      \bigwedge_{k=1}^n \Box^k\varphi.
    \end{gather*}
    If $\calF$ is $n$-transitive, then, for each $\varphi \in \mathrm{Fm}$ and valuation $\vartheta$
    on $\calF$,
    \begin{gather*}
      \vartheta(\Dmd^{+n}\varphi) = \{w \in W \mid R^+(w) \cap \vartheta(\varphi) \neq \emptyset\}.
    \end{gather*}
    It is easy to check that $\calF$ is $n$-transitive iff $\mathrm{Log}(\calF)$ is $n$-transitive.
    Therefore, all frames of $n$-transitive logic are $n$-transitive. Since $\mathrm{Trans}_n$ is
    canonical, the canonical frame of $n$-transitive logic is $n$-transitive.
  \subsection{Clusters, skeleton, maximal worlds, and Zorn's lemma}
    Let $\calF = (W, R)$ be a frame. Then $R^*$ is a \defw{preorder} (reflexive transitive relation)
    on $W$. The equivalence relation
    \begin{gather*}
      w \sim_R v :\Iff w \mathrel R^* v \wedge v \mathrel R^* w
    \end{gather*}
    is a congruence on $(W, R^*)$ (that is, if $w \sim_R w'$ and $v \sim_R v'$, then
    $w \mathrel R^* v \Iff w' \mathrel R^* v'$). The induced relation $\preceq_R$ on the quotient
    set $W/{\sim}_R$ is a partial order. The equivalence class $[w]_R$ of a world $w \in W$ under
    $\sim_R$ is called the \emph{cluster} of $w$. The poset $(W/{\sim}_R, \preceq_R)$ is called the
    \defw{skeleton} of $\calF$.

    Let $X$ be a subset of $W$. A world $v \in X$ is \emph{maximal in $X$}, if
    $v \mathrel R^* u \Rightarrow u \mathrel R^* v$ for all $u \in X$. We denote by
    $\mathrm{max}(X)$ the set of all worlds which are maximal in $X$. $S \subseteq W$ is a
    \emph{chain}, if, for all $w, v \in S$, $w \mathrel R^* v$ or $v \mathrel R^* w$. A world
    $v \in X$ is an \emph{upper bound of $S$}, if $u \mathrel R^* v$ for all $u \in S$. Note that
    $S = \emptyset$ is a chain and every $x \in X$ is its upper bound.
    \begin{lemma}[Zorn]
      If every chain in $X$ has an upper bound, then $\mathrm{max}(X) \neq \emptyset$.
    \end{lemma}
    \begin{proof}
      It is easy to derive this version of Zorn's lemma from the usual Zorn's lemma for posets
      applied to the skeleton of the frame $(X, R^*|_X)$.
    \end{proof}
  \subsection{Maximality property}
    The following property of the canonical model was established in~\cite{Fine85} for transitive
    logics and generalized to pretransitive case in~\cite{Shap21,KudShap25}.
    \begin{lemma}[maximality property]
      \label{l:max} Let $\Lambda$ be a pretransitive logic, $\Gamma$ be a set of formulas,
      \begin{gather*}
        X \defeq \{ w \in W_\Lambda \mid \calM_\Lambda, w \vDash \Gamma \}.
      \end{gather*}
      If $\Gamma$ is $\Lambda$-consistent, then $\mathrm{max}(X) \neq \emptyset$.
    \end{lemma}
    \begin{proof}
      Let $n$ be such that $\Lambda$ is $n$-transitive. Suppose that $\Gamma$ is
      $\Lambda$-consistent and $\mathrm{max}(X) = \emptyset$. Then, by Zorn's lemma, there is a
      chain $S$ in $X$ without an upper bound. Consider the set
      \begin{gather*}
        \Xi \defeq \{ \xi \mid \E{ w \in S}(\calM_\Lambda, w \vDash \Box^{+n}\xi) \}.
      \end{gather*}
      If the set $\Gamma \cup \Xi$ is $\Lambda$-consistent, then, by Lindenbaum's lemma, there is
      $v \in X$ such that $\calM_\Lambda, v \vDash \Xi$. By generalized tightness property,
      $w \mathrel R^+_\Lambda v$ for all $w \in S$. Therefore, $v \in X$ is an upper bound of $S$.
      Contradiction.

      If $\Gamma \cup \Xi$ is $\Lambda$-inconsistent, then
      $\Lambda \vdash \bigwedge_{i<r}\varphi_i \wedge \bigwedge_{j<s}\xi_j \impl \bot$ for some
      $\varphi_i \in \Gamma$, $\xi_j \in \Xi$. Note that $s > 0$, since $\Gamma$ is
      $\Lambda$-consistent. Let $w_j \in S$ be such that $\calM_\Lambda, w_j \vDash \Box^{+n}\xi_j$.
      Since $S$ is a chain, there is $k < s$ such that $w_j \mathrel R^*_\Lambda w_k$ for all
      $j < s$. Suppose that there is a world $u \in X \cap R^+_\Lambda(w_k)$. Then
      $w_j \mathrel R^+_\Lambda u$ for all $j < s$, therefore all $\varphi_i$ and $\xi_j$ are true
      in $u$, which is impossible, since $\calM_\Lambda \vDash \Lambda$. Thus, $w_k$ is maximal in
      $X$, contradicting our assumption.
    \end{proof}
    \begin{corollary}
      \label{cor:max} Let $\Lambda$ be a pretransitive logic, $w \in W_\Lambda$,
      $\psi \in \mathrm{Fm}$. If $\calM_\Lambda, w \vDash \Dmd\psi$, then
      \begin{gather*}
        \mathrm{max}\bigl(R_\Lambda(w) \cap \vartheta_\Lambda(\psi)\bigr) \neq \emptyset.
      \end{gather*}
    \end{corollary}
    \begin{proof}
      It is easy to see that the set $\Gamma \defeq \{ \varphi \in \mathrm{Fm} \mid \calM_\Lambda, w
      \vDash \Box\varphi\} \cup \{\psi\}$ is $\Lambda$-consistent. By tightness,
      \begin{gather*}
        R_\Lambda(w) \cap \vartheta_\Lambda(\psi) = \{ u \in W_\Lambda \mid \calM_\Lambda, u \vDash
        \Gamma\}.
      \end{gather*}
      Now it is sufficient to apply~\cref{l:max}.
    \end{proof}
 \section{Pretransitive analogues of K4, wK4, and GL}
  \label{s:class} Denote by $\mathrm{Fm}^+(p)$ the set of all formulas built from
  $p \in \mathrm{PVar}$ and the constant $\top$ using connectives $\wedge$ and $\Dmd$. Such formulas
  are called \defw{strictly positive}. For a formula $\varphi$, let
  $\underbar{d}_p(\varphi) \defeq \mathrm{min}(D_p(\varphi)\cup\{\infty\})$. In the rest of the
  work, we assume that $\beta, \gamma \in \mathrm{Fm}^+(p)$, $\beta$ and $\gamma$ contain $p$,
  $\underbar{d}_p(\beta) \geq 1$ and $\underbar{d}_p(\gamma) \geq 2$.

  Consider the logics $\mathrm{K4}_\gamma = \mathrm{K} + {\rm A}4_\gamma$,
  $\mathrm{wK4}_\gamma = \mathrm{K} + {\rm Aw}4_\gamma$, and
  $\mathrm{GL}_{\Dmd\beta} = \mathrm{K} + {\rm AL\ddot ob}_{\Dmd\beta}$, where
  \begin{align*}
    {\rm A}4_\gamma &= \gamma(p) \impl \Dmd p,\quad\\
    {\rm Aw}4_\gamma &= \gamma(p) \impl \Dmd p \vee p,\\
    {\rm AL\ddot ob}_{\Dmd\beta} &= \Dmd p \impl \Dmd(p \wedge \neg\beta(p)).
  \end{align*}
  Notice that this family of logics is slightly more general, then the family $C$ from the
  introduction. For example, it can be shown that there is no $\alpha \in \mathrm{Fm}^+(p)$ such
  that $\mathrm{K4}_{\Dmd^2\alpha} = \mathrm{K4}_{\Dmd^3p \wedge \Dmd^4p}$.

  \subsection{Basic relations between the logics under consideration}
    \begin{lemma}
      \label{l:in-wk4} If $\alpha \in \mathrm{Fm}^+(p)$ contains $p$, then
      $\mathrm{wK4} \vdash \alpha \impl \Dmd p \vee p$.
    \end{lemma}
    \begin{proof}
      We proceed by induction on construction of $\alpha$. For $\alpha \equiv \top$ the premise is
      false. For $\alpha \equiv p$, the statement is trivial.

      Suppose that $\alpha \equiv \alpha_1 \wedge \alpha_2$ contains $p$. Then, for some
      $i \in \{1, 2\}$, $\alpha_i$ contains $p$. By the induction hypothesis,
      $\mathrm{wK4} \vdash \alpha_i \impl \Dmd p \vee p$. Therefore,
      $\mathrm{wK4} \vdash \alpha \impl \Dmd p \vee p$.

      If $\alpha \equiv \Dmd\alpha_1$ contains $p$, then $\alpha_1$ contains $p$ and
      $\mathrm{wK4} \vdash \alpha_1 \impl \Dmd p \vee p$ by the induction hypothesis. By normality,
      $\mathrm{wK4} \vdash \Dmd\alpha_1 \impl \Dmd\Dmd p \vee \Dmd p$, therefore
      $\mathrm{wK4} \vdash \alpha \impl \Dmd p \vee p$.
    \end{proof}
    \begin{lemma}
      \label{l:in-k4} If $\alpha \in \mathrm{Fm}^+(p)$ contains $p$ and
      $\underbar{d}_p(\alpha) \geq 1$, then $\mathrm{K4} \vdash \alpha \impl \Dmd p$.
    \end{lemma}
    \begin{proof}
      The proof is by induction on construction of $\alpha$. For $\alpha \equiv \top$ and
      $\alpha \equiv p$, the premise is false.

      If $\alpha \equiv \alpha_1 \wedge \alpha_2$ contains $p$ and $\underbar{d}_p(\alpha) \geq 1$,
      then, for some $i \in \{1, 2\}$, $\alpha_i$ contains $p$ and
      $\underbar{d}_p(\alpha_i) \geq 1$. By the induction hypothesis,
      $\mathrm{K4} \vdash \alpha_i \impl \Dmd p$, whence $\mathrm{K4} \vdash \alpha \impl \Dmd p$.

      Suppose that $\alpha \equiv \Dmd\alpha_1$ contains $p$. By \cref{l:in-wk4},
      $\mathrm{wK4} \vdash \alpha_1 \impl \Dmd p \vee p$. By normality,
      $\mathrm{wK4} \vdash \Dmd\alpha_1 \impl \Dmd\Dmd p \vee \Dmd p$. Since
      $\mathrm{K4} \supseteq \mathrm{wK4}$ and $\mathrm{K4} \vdash \Dmd\Dmd p \impl \Dmd p$,
      $\mathrm{K4} \vdash \alpha \impl \Dmd p$.
    \end{proof}
    \begin{proposition}
      \label{pr:con} $\mathrm{(w)K4}_\gamma \subseteq \mathrm{(w)K4}$ and
      $\mathrm{GL}_{\Dmd\beta} \subseteq \mathrm{GL}$. In particular, $\mathrm{(w)K4}_\gamma$ and
      $\mathrm{GL}_{\Dmd\beta}$ are consistent.
    \end{proposition}
    \begin{proof}
      By \cref{l:in-wk4,l:in-k4},
      \begin{gather*}
        \mathrm{K4} \vdash \gamma \impl \Dmd p,\quad \mathrm{wK4} \vdash \gamma \impl \Dmd p \vee
        p,\quad\text{and}\quad \mathrm{K4} \vdash \beta \impl \Dmd p.
      \end{gather*}
      By monotonicity,
      $\mathrm{K4} \vdash \Dmd(p \wedge \neg\Dmd p) \impl \Dmd(p \wedge \neg\beta)$. Thus,
      $\mathrm{GL} \vdash \Dmd p \impl \Dmd(p \wedge \neg\beta)$.
    \end{proof}
    Consider the following sequence of formulas:
    \begin{gather*}
      \sigma_1 \defeq p,\quad \sigma_n \defeq p \wedge \Dmd\sigma_{n-1} \quad\text{for } n > 1.
    \end{gather*}
    Let $\calM = (W, R, \vartheta)$ be a Kripke model. It is easy to see that $\sigma_n$ is true at
    a world $v_1$ iff $\calM, v_1 \vDash p$ and there are worlds $v_2, \dots, v_n$ such that
    $v_{i-1} \mathrel R v_i$ and $\calM, v_i \vDash p$ for $i = 2, \dots, n$.
    \begin{lemma}
      \label{l:delta-n} If $\alpha \in \mathrm{Fm}^+(p)$, $\underbar{d}_p(\alpha) \geq k$, and
      $d(\alpha) < k + n$, then $\mathrm{K} \vdash \Dmd^k \sigma_n \impl \alpha$.
    \end{lemma}
    \begin{proof}
      We proceed by induction on construction of $\alpha$. For $\alpha \equiv \top$ the statement is
      trivial. If $\alpha \equiv p$, then $d(\alpha) = \underbar{d}_p(\alpha) = 0$ and
      $\mathrm{K} \vdash \Dmd^0 \sigma_n \impl \alpha$ for $n > 0$.

      If $\alpha \equiv \alpha_1 \wedge \alpha_2$, then, for $i = 1, 2$,
      \begin{gather*}
        \underbar{d}_p(\alpha_i) \geq \underbar{d}_p(\alpha) \geq k\quad\text{and}\quad d(\alpha_i)
        \leq d(\alpha) < k + n.
      \end{gather*}
      Therefore, $\mathrm{K} \vdash \Dmd^k \sigma_n \impl \alpha_1 \wedge \alpha_2$ by the induction
      hypothesis.

      If $\alpha \equiv \Dmd\alpha_1$, then
      \begin{gather*}
        \underbar{d}_p(\alpha_1) = \underbar{d}_p(\alpha) - 1 \geq k - 1\quad\text{and}\quad
        d(\alpha_1) = d(\alpha) - 1 < k - 1 + n.
      \end{gather*}
      By the induction hypothesis, $\mathrm{K} \vdash \Dmd^{k-1} \sigma_n \impl \alpha_1$, therefore
      $\mathrm{K} \vdash \Dmd^k\sigma_n \impl \alpha$ by monotonicity.
    \end{proof}
    \begin{proposition}
      \label{pr:delta-n} Suppose that $d(\gamma) \leq n + 1$ and $d(\beta) \leq n$. Then
      \begin{gather*}
        \mathrm{K4}_\gamma \supseteq \mathrm{K4}_{\Dmd^2 \sigma_n},\quad \mathrm{wK4}_\gamma
        \supseteq \mathrm{wK4}_{\Dmd^2 \sigma_n}, \quad\text{and}\quad \mathrm{GL}_{\Dmd\beta}
        \supseteq \mathrm{GL}_{\Dmd^2 \sigma_n}.
      \end{gather*}
    \end{proposition}
    \begin{proof}
      By \cref{l:delta-n}, the implications $\Dmd^2 \sigma_n \impl \gamma$ and
      $\Dmd\sigma_n \impl \beta$ are derivable in $\mathrm{K}$. Therefore, the implications
      ${\rm A}4_\gamma \impl {\rm A}4_{\Dmd^2\sigma_n}$,
      ${\rm Aw}4_\gamma \impl {\rm Aw}4_{\Dmd^2\sigma_n}$, and
      ${\rm AL\ddot ob}_{\Dmd\beta} \impl {\rm AL\ddot ob}_{\Dmd^2\sigma_n}$ are also derivable in
      $\mathrm{K}$.
    \end{proof}
    \begin{proposition}
      \label{pr:K4-subset-GL} $\mathrm{wK4}_\gamma \subseteq \mathrm{K4}_\gamma$ and
      $\mathrm{K4}_{\Dmd\beta} \subseteq \mathrm{GL}_{\Dmd\beta}$.
    \end{proposition}
    \begin{proof}
      The first inclusion is trivial. For the second, consider the formula
      $\varphi \defeq \beta(p) \vee p$. By induction on construction of $\beta$, one can easily
      check that $\mathrm{K} \vdash \beta(p) \impl \beta(\varphi)$, therefore
      \begin{gather}
        \mathrm{K} \vdash \varphi \impl \beta(\varphi) \vee p. \label{eq:u-1}
      \end{gather}
      Also, $\mathrm{K} \vdash \beta(p) \impl \varphi$, whence
      \begin{align*}
        \mathrm{GL}_\gamma \vdash \Dmd\beta(p) &\impl \Dmd\varphi
        &&\text{by monotonicity}\\
        &\impl \Dmd(\varphi \wedge \neg\beta(\varphi))
        &&\text{by }{\rm AL\ddot ob}_{\Dmd\beta}\\
        &\impl \Dmd\bigl((\beta(\varphi) \vee p) \wedge \neg\beta(\varphi)\bigr)
        &&\text{by \cref{eq:u-1}}\\
        &\impl \Dmd p.
      \end{align*}
    \end{proof}
    \begin{proposition}
      \label{pr:GL-frames} Let $\calF = (W, R)$ be a Kripke frame. $\mathrm{GL}_{\Dmd\beta}$ is
      valid in $\calF$ iff $\calF \vDash \mathrm{wK4}_{\Dmd\beta}$ and $R$ is conversely
      well-founded.
    \end{proposition}
    \begin{proof}
      Suppose that $R$ is not conversely well-founded. Then there is a sequence of worlds
      $V = \{ v_k \mid k \in \omega \}$ such that $v_k \mathrel R v_{k+1}$ for all $k \in \omega$.
      Let $\vartheta(p) \defeq V$, $n \defeq d(\beta)$. Clearly,
      $\vartheta(\Dmd\sigma_n) \supseteq V$, whence $\vartheta(\beta) \supseteq V$ by
      \cref{l:delta-n}. Therefore, $\vartheta(p) \cap \vartheta(\neg\beta) = \emptyset$ and
      $(\calF, \vartheta), v_1 \nvDash \Dmd p \impl \Dmd(p \wedge \neg\beta)$. Thus,
      $\mathrm{GL}_{\Dmd \beta}$ is not valid in $\calF$.

      Now let $\calF$ be a $\mathrm{wK4}_{\Dmd\beta}$-frame with conversely well-founded relation
      $R$. Suppose that, for some valuation $\vartheta$, ${\rm AL\ddot ob}_{\Dmd\beta}$ is false at
      $w \in W$ in the model $\calM = (\calF, \vartheta)$. Then $\Dmd p$ and $\Box(p \impl \beta)$
      are true at $w$. Consider the set $X \defeq R(w) \cap \vartheta(p)$. Since $R$ is conversely
      well-founded, there is $v \in X$ such that $X \cap R^+(v) = \emptyset$. Note that $\beta$ is
      true in $v$, since $\calM, w \vDash \Box(p \impl \beta)$.

      Consider the valuation $\vartheta'(p) \defeq \vartheta(p) \cap R^+(v)$. Since
      $\underbar{d}_p(\beta) \geq 1$, by \cref{l:modalized}, $\calM', v \vDash \beta$, therefore
      $\calM', w \vDash \Dmd\beta$ and, since $\mathrm{wK4}_{\Dmd\beta}$ is valid in $\calF$,
      $\calM', w \vDash \Dmd p \vee p$. $w \mathrel R v$ and $R$ is conversely well-founded, whence
      $w \notin R^+(v)$ and $\calM', w \nvDash p$. Thus, $\calM', w \vDash \Dmd p$ and the set
      \begin{gather*}
        R(w) \cap \vartheta'(p) = R(w) \cap \vartheta(p) \cap R^+(v) = X \cap R^+(v)
      \end{gather*}
      is not empty. Contradiction.
    \end{proof}
    Notice that formulas ${\rm A}4_\gamma$ are Sahlqvist and, therefore, have corresponding
    first-order conditions on Kripke frames. However, in general, these conditions are quite
    cumbersome. Therefore, we write them out only for the case $\gamma = \Dmd^2\sigma_n$:
    $\calF = (W, R)$ is a $\mathrm{K4}_{\Dmd^2\sigma_n}$-frame iff
    \begin{gather}
      \A{ w, u, v_1, \dots, v_n \in W} \left(w \mathrel R u \mathrel R v_1 \mathrel R \dots
      \mathrel R v_n \Rightarrow \bigvee_{i=1}^nw \mathrel R v_i\right). \label{eq:4-cond-local}
    \end{gather}
    \begin{proposition}
      \label{pr:inclusions} The following inclusions hold:
      \begin{center}
        \begin{tikzcd}[column sep = tiny]
          |[xshift=-.6cm]|\mathrm{GL} = \mathrm{GL}_{\Dmd^2\sigma_1} \arrow[d,symbol=\supset]
          \arrow[r,symbol=\supset]
          &[-1.4cm]|[xshift=-1.2cm]| \mathrm{GL}_{\Dmd^2\sigma_2} \arrow[d,symbol=\supset]
          \arrow[r,symbol=\supset] &[-1.4cm]|[xshift=-1.2cm]| \mathrm{GL}_{\Dmd^2\sigma_3}
          \arrow[d,symbol=\supset] \arrow[r,symbol=\supset] &[-1.4cm]|[xshift=-1.2cm]| \cdots
          \\[-8pt]
          \mathrm{K4} = \mathrm{K4}_{\Dmd^2\sigma_1} \arrow[d,symbol=\supset]
          \arrow[r,symbol=\supset]
          &[-.8cm]|[xshift=-.6cm]| \mathrm{K4}_{\Dmd^2\sigma_2} \arrow[d,symbol=\supset]
          \arrow[r,symbol=\supset] &[-.8cm]|[xshift=-.6cm]| \mathrm{K4}_{\Dmd^2\sigma_3}
          \arrow[d,symbol=\supset] \arrow[r,symbol=\supset] &[-.8cm]|[xshift=-.6cm]| \cdots \\[-8pt]
          |[xshift=.6cm]| \mathrm{wK4} = \mathrm{wK4}_{\Dmd^2\sigma_1} \arrow[ru,symbol=\supset]
          \arrow[r,symbol=\supset]
          &[-.8cm]|[xshift=.6cm]| \mathrm{wK4}_{\Dmd^2\sigma_2} \arrow[ru,symbol=\supset]
          \arrow[r,symbol=\supset] &[-.8cm]|[xshift=.6cm]| \mathrm{wK4}_{\Dmd^2\sigma_3}
          \arrow[ru,symbol=\supset] \arrow[r,symbol=\supset] &[-.8cm]|[xshift=.6cm]| \cdots
        \end{tikzcd}
      \end{center}
    \end{proposition}
    \begin{proof}
      Let us check that
      $\mathrm{wK4}_{\Dmd^2 \sigma_n} \supseteq \mathrm{K4}_{\Dmd^2 \sigma_{n+1}}$. Since
      $\mathrm{K} \vdash \sigma_{n+1}(p) \impl \sigma_n(p \wedge \Dmd p)$,
      \begin{align*}
        \mathrm{wK4}_{\Dmd^2\sigma_n} \vdash \Dmd^2 \sigma_{n+1}(p) &\impl \Dmd^2 \sigma_n(p \wedge
        \Dmd p)
        &&\text{by monotonicity}\\
        &\impl \Dmd(p \wedge \Dmd p) \vee p \wedge \Dmd p
        &&\text{by }{\rm Aw}4_{\Dmd^2\sigma_n}\\
        &\impl \Dmd p.
      \end{align*}
      The remaining non-strict inclusions follow from \cref{pr:delta-n} and \cref{pr:K4-subset-GL}.

      Now, we show that inclusions are strict. Consider the successor relation on natural numbers:
      $S \defeq \{(n, n+1) \mid n \in \omega\}$. For $n \in \omega$, let $\calF_n \defeq (n, S|_n)$,
      where we identify $n$ with the set $\{0, \dots, n-1\}$. It is easy to see that
      $\calF_{n+2} \vDash \neg\Dmd^2\sigma_{n+1}$ and $\calF_{n+2}$ is conversely well-founded.
      Therefore, $\calF_{n+2}$ is a $\mathrm{GL}_{\Dmd^2\sigma_{n+1}}$-frame by~\cref{pr:GL-frames}.
      At the same time, $\calF_{n+2} \nvDash \mathrm{wK4}_{\Dmd^2\sigma_n}$. Thus, inclusions
      between logics with different indices are strict.

      It remains to show that $\mathrm{GL}_{\Dmd^2\sigma_n} \supsetneq \mathrm{K4}_{\Dmd^2\sigma_n}
      \supsetneq \mathrm{wK4}_{\Dmd^2\sigma_n}$. For the first inclusion, consider a reflexive
      point: $\mathrm{GL}_{\Dmd^2\sigma_n}$ is not valid on it by~\cref{pr:GL-frames}, but it is a
      $\mathrm{K4}$-frame, hence $\mathrm{K4}_{\Dmd^2\sigma_n}$ is valid on it. For the second,
      consider the frames $\calF'_n \defeq (n, S|_n \cup \{n-1,0\})$. It is easy to see that
      $\calF'_n \vDash \Dmd^2\sigma_n \impl p$, but $\calF'_n \nvDash \Dmd^2\sigma_n \impl \Dmd p$.
      Therefore, $\calF'_n$ is a $\mathrm{wK4}_{\Dmd^2\sigma_n}$-frame, but not a
      $\mathrm{K4}_{\Dmd^2\sigma_n}$-frame.
    \end{proof}
    As we have shown, the logics $\mathrm{K4}_{\Dmd^2\sigma_n}$ are linearly ordered by inclusion.
    It is easy to check that the logics $\mathrm{K4}_{\Dmd^np}=\mathrm{K4}^1_n$ form a lattice by
    inclusion. More precisely,
    \begin{gather*}
      \mathrm{K4}^1_n \subseteq \mathrm{K4}^1_k \Iff (n-1) \text{ is divisible by } (k-1).
    \end{gather*}
    In general, the inclusion relation between logics $\mathrm{K4}_\gamma$ is more sophisticated. It
    is even not clear when two such logics coincide. Trivially,
    $\mathrm{K4}_{\gamma_1} = \mathrm{K4}_{\gamma_2}$ if $\mathrm{K} \vdash \gamma_1 \iff \gamma_2$.
    However, this condition is far from being necessary:
    $\mathrm{K4}_{\Dmd^2p \wedge \Dmd^3p} = \mathrm{K4}_{\Dmd^3p}$, despite the fact that
    $\Dmd^2p \wedge \Dmd^3p$ and $\Dmd^3p$ are not equivalent even in $\mathrm{K4}$.
    \begin{corollary}
      \label{cor:inclusions} The following logics contain $\mathrm{K4}_{\Dmd^2\sigma_n}$:
      \begin{gather*}
        \mathrm{K4}_\gamma, d(\gamma) \leq n + 1;\quad \mathrm{wK4}_\gamma, d(\gamma) \leq
        n;\quad\text{and}\quad \mathrm{GL}_\beta, d(\beta) \leq n.
      \end{gather*}
    \end{corollary}
    \begin{proof}
      Follows from \cref{pr:delta-n,pr:inclusions}.
    \end{proof}
    \begin{proposition}
      \label{pr:delta-n-trans} The logic $\mathrm{K4}_{\Dmd^2\sigma_n}$ is $n$-transitive.
    \end{proposition}
    \begin{proof}
      Consider the formula $\varphi \defeq \bigvee_{k=0}^{n-1} \Dmd^kp$. It is easy to check that
      \begin{gather*}
        \mathrm{K} \vdash \Dmd^{n+1}p \impl \Dmd^2\sigma_n(\varphi).
      \end{gather*}
      Therefore, $\mathrm{K4}_{\Dmd^2\sigma_n} \vdash \Dmd^{n+1}p \impl \Dmd\varphi$ and
      $\mathrm{K4}_{\Dmd^2\sigma_n} \vdash \Dmd^{n+1}p \impl \bigvee_{k=1}^n \Dmd^kp$ by normality.
    \end{proof}
    \begin{corollary}
      \label{cor:class-n-trans} The logics from \cref{cor:inclusions} are $n$-transitive.
    \end{corollary}
    \begin{lemma}
      \label{l:GL-trans-lob} For all $n \geq 1$,
      \begin{align*}
        \mathrm{K4}_{\Dmd^2\sigma_n} &\vdash \Dmd^{+n}\sigma_n(p) \impl \Dmd p,\\
        \mathrm{GL}_{\Dmd^2\sigma_n} &\vdash \Dmd p \impl \Dmd\bigl(p \wedge
        \neg\Dmd^{+n}\sigma_n(p)\bigr).
      \end{align*}
    \end{lemma}
    \begin{proof}
      One can easily check that, for all $k \geq 0$,
      \begin{gather}
        \mathrm{K} \vdash \Dmd^k \sigma_n(p) \impl \sigma_n\left(p \vee \bigvee_{l=1}^k \Dmd^l
        \sigma_n\right). \label{eq:tr-1}
      \end{gather}

      By induction on $k \geq 1$, we prove that
      \begin{gather}
        \mathrm{K4}_{\Dmd^2\sigma_n} \vdash \Dmd^k \sigma_n(p) \impl \Dmd p. \label{eq:tr-2}
      \end{gather}
      For $k = 1$, the formula is derivable in $\mathrm{K}$. Suppose that $k \geq 2$ and
      \cref{eq:tr-2} holds for $l = 1, \dots, k-1$. We have:
      \begin{align*}
        \mathrm{K4}_{\Dmd^2\sigma_n} \vdash \Dmd^k \sigma_n(p) &\impl \Dmd^2 \sigma_n\left(p \vee
        \bigvee_{l=1}^{k-2} \Dmd^l \sigma_n\right)
        &&\text{by }\cref{eq:tr-1}\\
        &\impl \Dmd\left(p \vee \bigvee_{l=1}^{k-2} \Dmd^l \sigma_n\right)
        &&\text{by }{\rm A}4_{\Dmd^2\sigma_n}\\
        &\impl \Dmd p \vee \bigvee_{l=2}^{k-1}\Dmd^l \sigma_n
        &&\text{by normality}\\
        &\impl \Dmd p
        &&\text{by the induction hypothesis}.
      \end{align*}

      The first statement of the lemma follows immediately from \cref{eq:tr-2}. Let us prove the
      second. From \cref{eq:tr-1}, it follows that
      \begin{gather}
        \mathrm{K} \vdash \Dmd^{+n} \sigma_n(p) \impl \Dmd\sigma_n(p \vee \Dmd^{+n}\sigma_n).
        \label{eq:tr-3}
      \end{gather}
      Thus,
      \begin{align*}
        \mathrm{GL}_{\Dmd^2\sigma_n} \vdash \Dmd p &\impl \Dmd(p \vee \Dmd^{+n}\sigma_n)\\
        &\impl \Dmd\bigl((p \vee \Dmd^{+n}\sigma_n) \wedge \neg\Dmd\sigma_n(p \vee
        \Dmd^{+n}\sigma_n)\bigr)
        &&\text{by }{\rm AL\ddot ob}_{\Dmd^2\sigma_n}\\
        &\impl \Dmd\bigl((p \vee \Dmd^{+n}\sigma_n) \wedge \neg\Dmd^{+n}\sigma_n\bigr)
        &&\text{by \cref{eq:tr-3}}\\
        &\impl \Dmd(p \wedge \neg\Dmd^{+n}\sigma_n).
      \end{align*}
    \end{proof}
    Transitivity of a relation $R$ can be defined either by local condition $R^2 \subseteq R$ or,
    equivalently, by the global one $R^+ \subseteq R$. Similar effects take place for the logics
    $\mathrm{K4}_{\Dmd^2\sigma_n}$: \cref{eq:4-cond-local} is the local condition, the global
    condition is stated in the following corollary:
    \begin{corollary}
      \label{cor:4-cond-global} A Kripke frame $\calF = (W, R)$ is a
      $\mathrm{K4}_{\Dmd^2\sigma_n}$-frame iff
      \begin{gather}
        \A{ w, v_1, \dots, v_n \in W} \left(w \mathrel R^+ v_1 \mathrel R \dots \mathrel R v_n
        \Rightarrow \bigvee_{i=1}^n w \mathrel R v_i\right). \label{eq:4-cond-global}
      \end{gather}
    \end{corollary}
    \begin{proof}
      Clearly, \cref{eq:4-cond-global} implies \cref{eq:4-cond-local}, therefore
      \cref{eq:4-cond-global} is a sufficient condition. Let us prove that it is necessary. Suppose
      that $\calF$ is a $\mathrm{K4}_{\Dmd^2\sigma_n}$-frame and
      $w \mathrel R^+ v_1 \mathrel R \dots \mathrel R v_n$. Consider the valuation
      $\vartheta(p) \defeq \{v_1, \dots, v_n\}$. Then $(\calF, \vartheta), v_1 \vDash \sigma_n$ and,
      since $\calF$ is $n$-transitive, $(\calF, \vartheta), w \vDash \Dmd^{+n}\sigma_n$.
      By~\cref{l:GL-trans-lob}, $(\calF, \vartheta), w \vDash \Dmd p$. Thus, $w \mathrel R v_i$ for
      some $i \in \{1, \dots, n\}$
    \end{proof}
    Notice that, the equality $\mathrm{K4}_{\Dmd\beta} = \mathrm{K} + \Dmd^{+n}\beta \impl \Dmd p$
    does not hold in general. For example, $\mathrm{K4}_{\Dmd^3p} \vdash \Dmd^k p \impl \Dmd p$ iff
    $k$ is odd.

  \subsection{Subframe- and semisubframe-hereditary logics}
    \begin{definition}
      A weak subframe $\calF = (W, R)$ of the frame $\calF_0 = (W_0, R_0)$ is a \defw{semisubframe}
      if the following condition holds:
      \begin{gather}
        \forall w, v \in W \,(w \mathrel R^* v \wedge w \mathrel R_0 v \Rightarrow w \mathrel R v).
        \label{eq:semisubframe}
      \end{gather}
      A model $(\calF, \vartheta)$ is a \defw{semisubmodel} of $(\calF_0, \vartheta_0)$, if $\calF$
      is a semisubframe of $\calF_0$ and $\vartheta = \vartheta_0|_W$.
    \end{definition}
    In contrast to generated subframes, the class of frames of an arbitrary logic is not neccessary
    closed under taking subframes and semisubframes. We say that a Kripke-complete logic $\Lambda$
    is \defw{(semi)subframe-hereditary}, if the class $\mathrm{Frames}(\Lambda)$ is closed under
    taking (semi)subframes. A formula $\varphi$ is called \emph{(semi)subframe-hereditary} if
    $\mathrm{K} + \varphi$ is (semi)subframe-hereditary. Notice that, if $\varphi$ is canonical,
    then $\mathrm{K} + \varphi$ is Kripke-complete and this definition coincides with the definition
    from the introduction. It is obvious that any subframe is a semisubframe, therefore all
    semisubframe-hereditary logics are subframe-hereditary.

    Some comments are needed here. Firstly, subframe-hereditary logics are more often called simply
    \emph{subframe} (from the classical work of Fine~\cite{Fine85} to modern
    papers~\cite{BGJ11,KudShap25}). However, the term \emph{subframe formula} is commonly used for
    specific formulas defined by Fine. In this paper, we use more precise (though longer) term
    \emph{subframe-hereditary} for both formulas and logics. Secondly, the notion of a
    subframe-hereditary logic can be extended to Kripke-incomplete logics using general
    frames~\cite[Section 2.2]{Wolt93}.

    Let us give some examples of subframe- and semisubframe-hereditary formulas. It is easy to see
    that every semisubframe of a transitive (reflexive) frame is transitive (reflexive), whence the
    formulas $\Dmd^2p \impl \Dmd p$ and $p \impl \Dmd p$ are semisubframe-hereditary. Also, every
    subframe (but not every semisubframe) of a symmetric frame is symmetric. Therefore,
    $p \impl \Box\Dmd p$ is a subframe-, but not semisubframe-hereditary formula. Finally, a
    subframe of a serial frame can easily be not serial, so $\Dmd\top$ is not a subframe-hereditary
    formula.
    \begin{proposition}
      \label{pr:t-inherit} The logics $\mathrm{K4}_\gamma$, $\mathrm{wK4}_\gamma$, and
      $\mathrm{GL}_{\Dmd\beta}$ are semisubframe-hereditary.
    \end{proposition}
    \begin{proof}
      Suppose that there is a semisubframe $\calF = (W, R)$ of the frame $\calF_0 = (W_0, R_0)$ such
      that $\calF \nvDash \mathrm{K4}_\gamma$. Then there is a valuation $\vartheta$ on $\calF$ and
      a world $w \in W$ such that $(\calF, \vartheta), w \nvDash \gamma \impl \Dmd p$. By
      \cref{l:modalized}, we can assume that $\vartheta(p) \subseteq R^*(w)$.

      Let $\vartheta_0$ be the valuation on $\calF_0$ such that $\vartheta_0(p) = \vartheta(p)$. By
      induction on construction of $\alpha \in \mathrm{Fm}^+(p)$, we will show that
      $\vartheta_0(\alpha) \supseteq \vartheta(\alpha)$. The base case is trivial. Suppose that
      $\alpha \equiv \Dmd\alpha_1$. By the induction hypothesis,
      $\vartheta_0(\alpha_1) \supseteq \vartheta(\alpha_1)$. Since $W_0 \supseteq W$ and
      $R_0 \supseteq R$,
      \begin{align*}
        \vartheta_0(\Dmd\alpha_1) &= \{ v \in W_0 \mid R_0(v) \cap \vartheta_0(\alpha_1) \neq
        \emptyset \}\\
        &\supseteq \{ v \in W \mid R(v) \cap \vartheta(\alpha_1) \neq \emptyset \} =
        \vartheta(\Dmd\alpha_1).
      \end{align*}
      The induction step for $\alpha \equiv \alpha_1 \wedge \alpha_2$ is trivial.

      Therefore, $(\calF_0, \vartheta_0), w \vDash \gamma$. Moreover,
      \begin{align*}
        R_0(w) \cap \vartheta_0(p) &= R_0(w) \cap R^*(w) \cap \vartheta(p) &&\text{since
        $\vartheta_0(p) = \vartheta(p) \subseteq R^*(w)$}\\
        &= R(w) \cap \vartheta(p) &&\text{by \cref{eq:semisubframe}}.
      \end{align*}
      Thus, $(\calF_0, \vartheta_0), w \nvDash \Dmd p$ and $\calF_0 \nvDash \mathrm{K4}_\gamma$.

      So, $\mathrm{K4}_\gamma$ is semisubframe-hereditary. Notice that, if
      $w \notin \vartheta({\rm Aw}4_\gamma)$, then $w \notin \vartheta(p) = \vartheta_0(p)$, whence
      $w \notin \vartheta_0({\rm Aw}4_\gamma)$, therefore $\mathrm{wK4}_\gamma$ is also
      semisubframe-hereditary. One can easily show that the class of all conversely well-founded
      frames is closed under taking semisubframes, whence, by \cref{pr:GL-frames},
      $\mathrm{GL}_{\Dmd\beta}$ is semisubframe-hereditary.
    \end{proof}
    Formulas ${\rm A}4_\gamma$ and ${\rm Aw}4_\gamma$ are Sahlqvist, therefore $\mathrm{K4}_\gamma$
    and $\mathrm{wK4}_\gamma$ are canonical. In contrast, $\mathrm{GL}_{\Dmd\beta}$ is not
    canonical. Indeed, by \cref{pr:con} $\mathrm{GL}_{\Dmd\beta} \subseteq \mathrm{GL}$, whence
    $\calF_\mathrm{GL}$ is a generated subframe of $\calF_{\mathrm{GL}_{\Dmd\beta}}$ by
    \cref{l:canon-subframe}. It is known that $\calF_\mathrm{GL}$ contains reflexive worlds
    \cite[Theorem 6.5]{ChZa97}, therefore $\calF_{\mathrm{GL}_{\Dmd\beta}}$ also contains them and
    $\calF_{\mathrm{GL}_{\Dmd\beta}} \nvDash \mathrm{GL}_{\Dmd\beta}$ by~\cref{pr:GL-frames}.
    \begin{corollary}
      \label{cor:canon-subframes} $\mathrm{K4}_\gamma$ and $\mathrm{wK4}_\gamma$ are valid in all
      semisubframes of their canonical frames. $\mathrm{GL}_{\Dmd\beta}$ is valid in all conversely
      well-founded semisubframes of its canonical frame.
    \end{corollary}
    \begin{proof}
      The first claim follows from canonicity and \cref{pr:t-inherit}. For the second, let $\calF$
      be a conversely well-founded semisubframe of $\calF_{\mathrm{GL}_{\Dmd\beta}}$. By
      \cref{pr:K4-subset-GL,c:canon-subframe},
      $\calF_{\mathrm{GL}_{\Dmd\beta}} \vDash \mathrm{K4}_{\Dmd\beta}$, therefore
      $\calF \vDash \mathrm{K4}_{\Dmd\beta}$ by~\cref{pr:t-inherit} and
      $\calF \vDash \mathrm{GL}_{\Dmd\beta}$ by~\cref{pr:GL-frames}.
    \end{proof}
  \subsection{The main results}
    \label{ss:main} Now we are ready to state the main results of the work.
    \begin{theorem}
      \label{t:K4-fmp} Suppose that a logic $\Lambda$ is valid in finite subframes of its canonical
      frame and contains $\mathrm{K4}_{\Dmd^2\sigma_n}$ for some $n \in \omega$. Then it has the
      finite model property.
    \end{theorem}
    \begin{theorem}
      \label{t:GL-fmp} Suppose that $\Lambda$ is valid in finite conversely well-founded
      semisubframes of its canonical frame and contains $\mathrm{GL}_{\Dmd^2\sigma_n}$ for some
      $n \in \omega$. Then it has the finite model property.
    \end{theorem}
    By \cref{cor:inclusions,cor:canon-subframes}, conditions of the theorems hold for
    $\mathrm{K4}_\gamma$, $\mathrm{wK4}_\gamma$, and $\mathrm{GL}_{\Dmd\beta}$.
    \begin{corollary}
      \label{cor:fmp} The logics $\mathrm{K4}_\gamma$, $\mathrm{wK4}_\gamma$, and
      $\mathrm{GL}_{\Dmd\beta}$ have the finite model property.
    \end{corollary}
    Unfortunately, it is difficult to provide other natural examples of non-transitive modal logics
    for which these theorems can be applied. It was noted above that $p \impl \Dmd p$ and
    $p \impl \Box\Dmd p$ are subframe-hereditary. However, it is easy to see that all reflexive
    $\mathrm{K4}_{\Dmd^2\sigma_n}$-frames are transitive, whence
    $\mathrm{K4}_\gamma + p \impl \Dmd p = \mathrm{K4} + p \impl \Dmd p$ by~\cref{cor:inclusions}.
    Similar trivialization occur for symmetry, linearity, and other standard subframe-hereditary
    axioms of modal logics. In the context of pretransitive logics, it is more natural to consider
    their version for the transitive closure: $p \impl \Dmd^{+n} p$, $p \impl \Box^{+n}\Dmd^{+n} p$
    (logics with the second axiom were considered in~\cite{KudShap11}). However, it is easy to see
    that these formulas are not subframe-hereditary.

    Anyway, we can provide an example of modal logic other than $\mathrm{K4}_\gamma$ and
    $\mathrm{wK4}_\gamma$ for which \cref{t:K4-fmp} can be applied. Let us fix some variable $s$ and
    introduce a new modality $\Dmd_s\varphi \defeq \Dmd(s \wedge \varphi)$. It is easy to see that
    the formula
    \begin{gather*}
      \mathrm{Trans}_n^s \defeq s \impl \left(\Dmd_s^{n+1}p \impl \bigvee_{k=1}^n \Dmd_s^kp\right)
    \end{gather*}
    is valid in a frame $\calF$ iff $\calF$ and all its subframes are n-transitive. Therefore,
    $\mathrm{Trans}_n^s$ is clearly subframe-hereditary. It is also Sahlqvist, whence canonical.
    Thus, the logics $\mathrm{K4}_\gamma + \mathrm{Trans}_n^s$ have the finite model property
    by~\cref{t:K4-fmp}. One can easily check that, if $n < d(\gamma) - 1$, then $\mathrm{K4}_\gamma$
    is not $n$-transitive and $\mathrm{K4}_\gamma + \mathrm{Trans}_n^s$ is a nontrivial (though not
    very natural) extension of $\mathrm{K4}_\gamma$. The logic $\mathrm{K} + \mathrm{Trans}_n^s$ is
    much more natural, but it does not contain $\mathrm{K4}_{\Dmd^2\sigma_m}$ for any $m$,
    so~\cref{t:K4-fmp} is not applicable to it.
 \section{Paths in frames and models}
  \label{s:grid} In this section we prove auxiliary combinatorial results about
  $\mathrm{K4}_{\Dmd^2\sigma_n}$-frames.
  \subsection{Paths in frames}
    \begin{definition}
      Let $\calF = (W, R)$ be a Kripke frame. A finite sequence of worlds $u_0, \dots, u_m$ is a
      \defw{path} in $\calF$, if $u_k \mathrel R u_{k+1}$ for all $k < m$. The number $m \in \omega$
      is called the \defw{length} of the path.
    \end{definition}
    Let us fix some $n \in \omega$ and put $N \defeq n^n$.
    \begin{proposition}
      \label{pr:connected-paths} Let $\calF$ be a $\mathrm{K4}_{\Dmd^2\sigma_n}$-frame. Suppose that
      $u^i_0, \dots, u^i_n$, $i = 0, \dots, N$ are paths of length $n$ in $\calF$ such that
      $u^i_n \mathrel R^* u^{i+1}_0$ for all $i < N$. Then there are $i, i' \leq N$ and $j < n$ such
      that $i < i'$ and $u^i_j \mathrel R u^{i'}_{j+1}$.
    \end{proposition}
    \begin{proof}
      Suppose that $u^i_j \mathrel R u^{i'}_{j+1}$ does not hold for all $i < i' \leq N$ and
      $j < n$. We say that a pair of sequences of natural numbers $(i_1, \dots, i_l)$,
      $(j_2, \dots, j_l)$ is an \defw{$l \times m$-zigzag}, if $1 \leq i_1 < \dots < i_l \leq N$,
      $1 \leq j_2, \dots, j_l \leq m$, and $u^{i_k}_{m} \mathrel R u^{i_{k+1}}_{j_{k+1}}$ for all
      $k = 1, \dots, l-1$. The number $i_1$ is called the \defw{initial line} of the zigzag. Notice
      that, for each zigzag, there is the corresponding path in $\calF$:
      \begin{center}
        \begin{tikzcd}[column sep = small,row sep = .35em]
          u^{i_1}_1 \arrow[r]& u^{i_1}_2 \arrow[r]& \dots \arrow[r]& u^{i_1}_m \arrow[lld] \\
          & u^{i_2}_{j_2} \arrow[r]& \dots \arrow[r]& u^{i_2}_m \arrow[lld] \\ &
          \phantom{u^{i_3}_{j_3}} & \dots & \phantom{u^{i_3}_m} \arrow[lld]\\ & u^{i_l}_{j_l}
          \arrow[r]& \dots \arrow[r]& u^{i_l}_m.
        \end{tikzcd}
      \end{center}

      Let us note that, for $i = 1, \dots, N-1$,
      \begin{gather*}
        u^i_m \mathrel R^+ u^{i+1}_1 \mathrel R u^{i+1}_2 \mathrel R \dots \mathrel R u^{i+1}_n.
      \end{gather*}
      Therefore, $u^i_m \mathrel R u^{i+1}_{j_{i+1}}$ for some $j_{i+1} \in \{1, \dots, n\}$ by
      \cref{cor:4-cond-global}, whence $(1, 2, \dots, N)$, $(j_2, \dots, j_N)$ is an
      $n^n \times n$-zigzag.

      Suppose that there is an $n \times 1$-zigzag $(i_1, \dots, i_n)$, $(j_2, \dots, j_n)$. Then
      $j_2 = \dots = j_n = 1$ and
      \begin{gather*}
        u^0_0 \mathrel R^+ u^{i_1}_1 \mathrel R u^{i_2}_1 \mathrel R \dots \mathrel R u^{i_n}_1.
      \end{gather*}
      By \cref{cor:4-cond-global}, $u^0_0 \mathrel R u^{i_a}_1$ for some $a \in \{1, \dots, n\}$
      contradicting our assumption.

      So, we know that $n^n \times n$-zigzag exists and $n \times 1$-zigzag does not exist.
      \begin{lemma}
        \label{l:zigzag} If there is a $kn \times m$-zigzag with initial line $i_1$, then there is a
        $k \times (m-1)$-zigzag with the same initial line.
      \end{lemma}
      \begin{proof}
        The proof is by induction on $k$. The base case $k = 1$ is trivial. Suppose that $k > 1$ and
        the statement holds for $k - 1$. Let $(i_1, \dots, i_{kn})$, $(j_2, \dots, j_{kn})$ be a
        $kn \times m$-zigzag.

        Consider the path corresponding to the zigzag, more specifically, its fragment located on
        lines $i_2, \dots, i_{n+1}$:
        \begin{center}
          \newcommand{\ph}{\phantom{{}_{{}_{n+{}}} }}
          \begin{tikzcd}[column sep = small,row sep = .35em]
            \ph u^{i_2}_{j_2} \arrow[r]& \dots \arrow[r]& u^{i_2}_m \ph \arrow[lld] \\ \ph
            u^{i_3}_{j_3} \arrow[r]& \dots \arrow[r]& u^{i_3}_m \ph \arrow[lld] \\
            \phantom{u^{i_{n+1}}_{j_{n+1}}} & \dots & \phantom{u^{i_{n+1}}_m} \arrow[lld]\\
            u^{i_{n+1}}_{j_{n+1}} \arrow[r]& \dots \arrow[r]& u^{i_{n+1}}_m.
          \end{tikzcd}
        \end{center}
        There are at least $n$ worlds in this fragment (at least one in each line) and
        $u^{i_1}_{m-1} \mathrel R^+ u^{i_2}_{j_2}$, therefore, by \cref{cor:4-cond-global},
        $u^{i_1}_{m-1} \mathrel R u^{i_a}_j$ for some $a \in \{2, \dots, n+1\}$,
        $j \in \{j_a, \dots, m\}$. Moreover, by assumption, $j$ can not be equal to $m$.

        Consider the part of the zigzag starting from the $a$-th line:
        \begin{gather*}
          (i_a, i_{a+1}, \dots, i_{kn}),\; (j_{a+1}, \dots, j_{kn})
        \end{gather*}
        It is obviously an $l \times m$-zigzag, where $l = kn - a + 1 \geq (k - 1)n$. By the
        induction hypothesis, there is an $(k - 1) \times (m - 1)$-zigzag
        $(i'_1, i'_2, \dots, i'_{k-1})$, $(j'_2, \dots, j'_{k-1})$ such that $i'_1 = i_a$.

        It remains to notice that
        \begin{gather*}
          (i_1, i_a, i'_2, \dots, i'_{k-1}),\; (j, j'_2, \dots, j'_{k-1})
        \end{gather*}
        is an $k \times (m - 1)$-zigzag.
      \end{proof}
      Applying $n-1$ times \cref{l:zigzag} to the $n^n \times n$-zigzag, we obtain an
      $n \times 1$-zigzag. Contradiction.
    \end{proof}
  \subsection{Labeled paths in models}
    \begin{definition}
      Let $\calM = (W, R, \vartheta)$ be a Kripke frame, $\Psi$ be a finite set of formulas.
      \emph{$\Psi$-labeled path} in $\calM$ is a finite sequence of worlds and formulas
      \begin{gather*}
        u_0, \psi_0, u_1, \psi_1, \dots, \psi_{m-1}, u_m
      \end{gather*}
      in which $m \geq 0$, $u_k \in W$, $\psi_k \in \Psi$ and
      $u_{k+1} \in R(u_k) \cap \vartheta(\psi_k)$ for all $k < m$. The number $m$ is called the
      \defw{length} of this path.
    \end{definition}

    Let us fix some $\Psi$ and put $M \defeq N|\Psi|^n = n^n|\Psi|^n$.
    \begin{lemma}
      \label{l:grid-in-frame} Let $\calM = (W, R, \vartheta)$ be a model on a
      $\mathrm{K4}_{\Dmd^2\sigma_n}$-frame. Suppose that $u^i_0, \psi^i_0, \dots, u^i_n$ for
      $i = 0, \dots, M$ are $\Psi$-labeled paths of length $n$ in $\calM$ such that
      $u^i_n \mathrel R^* u^{i+1}_0$ for all $i < M$. Then there are $i, i' \leq M$ and $j < n$ such
      that $i < i'$ and
      \begin{gather}
        \label{eq:grid-in-frame}
        u^{i'}_{j+1} \in R(u^i_j) \cap \vartheta(\psi^i_j).
      \end{gather}
    \end{lemma}
    \begin{proof}
      Consider the vectors of formulas
      \begin{gather*}
        \vec{\psi\,}_i = (\psi^i_0, \dots, \psi^i_{n-1}),\quad i = 0, \dots, M.
      \end{gather*}
      There are only $|\Psi|^n$ distinct vectors of formulas from $\Psi$ of length $n$. Therefore,
      by the pigeonhole principle, there are $n^n + 1$ identical vectors among $\vec{\psi\,}_i$,
      $i = 0, \dots, M$, that is, $\vec{\psi\,}_{i_0} = \dots = \vec{\psi\,}_{i_N}$ for some
      $0 \leq i_0 < \dots < i_N \leq M$. Consider the paths
      \begin{gather*}
        u^{i_k}_0, \dots, u^{i_k}_n,\quad k = 0, \dots, N.
      \end{gather*}
      By~\cref{pr:connected-paths}, there are $k, k' \leq N$ and $j < n$ such that $k < k'$ and
      $u^{i_k}_j \mathrel R u^{i_{k'}}_{j+1}$. Since
      $u^{i_{k'}}_0, \psi^{i_{k'}}_0, \dots, u^{i_{k'}}_n$ is a labeled path and
      $\vec{\psi\,}_{i_{k'}} = \vec{\psi\,}_{i_k}$,
      $u^{i_{k'}}_{j+1} \in \vartheta(\psi^{i_{k'}}_j) = \vartheta(\psi^{i_k}_j)$. Thus,
      \cref{eq:grid-in-frame} holds for $i = i_k$ and $i' = i_{k'}$.
    \end{proof}
    \begin{definition}
      A $\Psi$-labeled path $u_0, \psi_0, \dots, u_m$ in $\calM = (W, R, \vartheta)$ is
      \emph{optimal}, if, for all $k < m$, $u_{k+1} \in \mathrm{max}(R(u_k) \cap \vartheta(\psi_k)$.
    \end{definition}
    \begin{corollary}
      \label{cor:seq-out} Let $\calM = (W, R, \vartheta)$ be a model on a
      $\mathrm{K4}_{\Dmd^2\sigma_n}$-frame, $l$ be a number greater than or equal to $n$. Suppose
      that $u_0, \psi_0, \dots, u_{lM + n}$ is an optimal $\Psi$-labeled path of length $lM + n$ in
      $\calM$. Then there is $k \leq l(M - 1) + n$ such that $u_{k+l} \mathrel R^+ u_k$.
    \end{corollary}
    \begin{proof}
      Let $u^i_j \defeq u_{li + j}$ for $i \leq M$, $j \leq n$, $\psi^i_j \defeq \psi_{li + j}$ for
      $i \leq M$, $j < n$. Clearly, $u^i_0, \psi^i_0, \dots, u^i_n$, $i \leq M$, are $\Psi$-labeled
      paths in $\calM$. Also, since $n \leq l$, $li + n \leq l(i + 1)$, therefore
      $u^i_n \mathrel R^* u^{i+1}_0$ for all $i < M$. By~\cref{l:grid-in-frame}, there are
      $i, i' \leq M$, $j < n$ such that $i < i'$ and
      $u^{i'}_{j+1} \in R(u^i_j) \cap \vartheta(\psi^i_j)$. Since
      $u^i_{j+1} \in \mathrm{max}(R(u^i_j) \cap \vartheta(\psi^i_j))$ and
      $u^i_{j+1} \mathrel R^+ u^{i'}_{j+1}$, $u^{i'}_{j+1} \mathrel R^+ u^i_{j+1}$. Thus,
      $u^{i+1}_{j+1} \mathrel R^+ u^i_{j+1}$, that is, $u_{k+l} \mathrel R^+ u_k$ for
      $k \defeq li + j + 1$.
    \end{proof}
    \begin{definition}
      A $\Psi$-labeled path $u_0, \psi_0, \dots, u_m$ in $\calM = (W, R, \vartheta)$ is called
      \emph{reducible}, if there are $k, k' < m$ such that $k' \leq k$ and
      $u_{k'} \in R(u_k) \cap \vartheta(\psi_k)$.
    \end{definition}
    \begin{corollary}
      \label{cor:seq-in} Let $\calM = (W, R, \vartheta)$ be a model on a
      $\mathrm{K4}_{\Dmd^2\sigma_n}$-frame. Then all $\Psi$-labeled paths of length $n(M + 1)$ in
      $\calM$ in which all worlds belong to the same cluster are reducible.
    \end{corollary}
    \begin{proof}
      Suppose that $u_0, \psi_0, \dots, u_{n(M + 1)}$ is a path in $\calM$ in which all worlds
      belong to the same cluster. Let $u^i_j \defeq u_{n(M-i)+j}$ for $i \leq M$, $j \leq n$,
      $\psi^i_j \defeq \psi_{n(M-i)+j}$ for $i \leq M$, $j < n$. Obviously,
      $u^i_0, \psi^i_0, \dots, u^i_n$, $i \leq M$, are $\Psi$-labeled paths in $\calM$. Since all
      worlds are in one cluster, $u^i_n \mathrel R^* u^{i+1}_0$ for all $i < M$.
      By~\cref{l:grid-in-frame}, there are $i, i' \leq M$, $j < n$ such that $i < i'$ and
      $u^{i'}_{j+1} \in R(u^i_j) \cap \vartheta(\psi^i_j)$. Let $k \defeq n(M - i) + j$,
      $k' \defeq n(M - i') + j + 1$. Then $k < nM + n$, $u_{k'} \in R(u_k) \cap \vartheta(\psi_k)$,
      and $k - k' = n(i' - i) - 1 \geq n - 1 \geq 0$, whence $k' \leq k$.
    \end{proof}
    \begin{corollary}
      \label{cor:K4} Let $\calM = (W, R, \vartheta)$ be a model on a
      $\mathrm{K4}_{\Dmd^2\sigma_n}$-frame. Then every optimal $\Psi$-labeled path of length
      $C \defeq n(M^2 + M + 1)$ in $\calM$ is reducible.
    \end{corollary}
    \begin{proof}
      Let $u_0, \psi_0, \dots, u_C$ be an optimal $\Psi$-labeled path in $\calM$. Notice that
      $C = lM + n$ for $l \defeq n(M + 1)$. By~\cref{cor:seq-out}, there is $k \leq l(M - 1) + n$
      such that $u_{k+l} \mathrel R^+ u_k$. Then the worlds $u_k, \dots, u_{k+l}$ belong to the same
      cluster. By~\cref{cor:seq-in}, $u_k, \psi_k, \dots, u_{k+l}$ is reducible. Thus,
      $u_0, \psi_0, \dots, u_C$ is also reducible.
    \end{proof}

  \section{Finite model property}
    \label{s:fmp} In this section, we prove \cref{t:K4-fmp,t:GL-fmp}. In both cases we use the
    method of selective filtration of the canonical model, maximality property, and combinatorial
    properties which were established in the previous section.

    \subsection{Pretransitive analogues of K4 and wK4}
  \label{s:K4-fmp} Let us prove~\cref{t:K4-fmp}. Suppose that $\Lambda$ is valid in all finite
  subframes of $\calF_\Lambda$ and contains $\mathrm{K4}_{\Dmd^2\sigma_n}$ for some $n \geq 1$. Let
  $\zeta$ be a formula non-derivable in $\Lambda$. Then $\{\neg\zeta\}$ is $\Lambda$-consistent and,
  by Lindenbaum's lemma, there is a world $x \in \vartheta_\Lambda(\neg\zeta)$. We construct the
  sets $W_k \subseteq W_\Lambda$, $k \in \omega$ by recursion. Let us put $W_0 \defeq \{x\}$.

  Suppose that, for some $k \in \omega$, $W_k$ is defined. Consider the set
  \begin{gather*}
    D_k \defeq \{(w, \psi) \in W_k \times \Psi^\zeta \mid \calM_\Lambda, w \vDash \Dmd\psi \}.
  \end{gather*}
  For $(w, \psi) \in D_k$, let us denote
  \begin{gather*}
    V^b_k(w, \psi) \defeq R_\Lambda(w) \cap \vartheta_\Lambda(\psi) \cap
    \bigcup_{l\leq k}W_l\quad\text{and}\quad V^f_k(w, \psi) \defeq \mathrm{max}\bigl(R_\Lambda(w)
    \cap \vartheta_\Lambda(\psi)\bigr)
  \end{gather*}
  (b and f mean backward and forward respectively). Notice that the set $V^f_k(w, \psi)$ is nonempty
  by~\cref{cor:max}. Then we put
  \begin{gather*}
    D^b_k \defeq \{(w, \psi) \in D_k \mid V^b_k(w, \psi) \neq \emptyset\}\quad\text{and}\quad D^f_k
    \defeq D_k \setminus D^b_k.
  \end{gather*}
  Let us fix some elements $v_k(w, \psi) \in V^b_k(w, \psi)$ for $(w, \psi) \in D^b_k$ and
  $v_k(w, \psi) \in V^f_k(w, \psi)$ for $(w, \psi) \in D^f_k$. Now, we put
  \begin{gather*}
    W_{k+1} \defeq \{ v_k(w, \psi) \mid (w, \psi) \in D^f_k \}.
  \end{gather*}

  Finally, let $W \defeq \bigcup_{k \in \omega}W_k$, $R \defeq R_\Lambda|_W$, and
  $\vartheta \defeq \vartheta_\Lambda|_W$. It is easy to see from the construction that
  $\calM \defeq (W, R, \vartheta)$ is a $\zeta$-selective submodel of $\calM_\Lambda$. Therefore,
  $\calM, x \nvDash \zeta$ by~\cref{l:filtration}. Also, notice that $|W_0| = 1$ and
  $|W_{k+1}| \leq |D_k| \leq |W_k||\Psi^\zeta|$. Therefore, $|W_k| \leq |\Psi^\zeta|^k$ and
  $|\bigcup_{l < k}W_l| \leq \sum_{l < k}|\Psi^\zeta|^l < |\Psi^\zeta|^k$.

  By induction on $k \in \omega$ we will show that, if $v \in W_k$, then there is an irreducible
  optimal $\Psi^\zeta$-labeled path $u_0, \psi_0, \dots, u_k$ of length $k$ in $\calM_\Lambda$ such
  that $u_k = v$ and $u_l \in W_l$ for $l < k$ (in particular, $u_0 = x$). The base case is trivial.
  Suppose that the statement holds for $k \in \omega$ and $v \in W_{k+1}$. Then, by construction,
  $v = v_k(w, \psi)$ for some $(w, \psi) \in D^f_k$, whence $w \in W_k$,
  $V^b_k(w, \psi) = \emptyset$, and $v \in \mathrm{max}(R_\Lambda(w) \cap \vartheta_\Lambda(\psi))$.
  By the induction hypothesis, there is an irreducible optimal $\Psi^\zeta$-labeled path
  $u_0, \psi_0, \dots, u_k$ such that $u_k = w$ and $u_l \in W_l$ for $l < k$. Let
  $\psi_k \defeq \psi$, $u_{k+1} \defeq v$. Since
  $v \in \mathrm{max}(R_\Lambda(w) \cap \vartheta_\Lambda(\psi))$, the labeled path
  $u_0, \psi_0, \dots, u_{k+1}$ is optimal. Let us show that it is irreducible. Suppose that there
  are $l, l' < k + 1$ such that $l' \leq l$ and
  $u_{l'} \in R_\Lambda(u_l) \cap \vartheta_\Lambda(\psi_l)$. Since $u_0, \psi_0, \dots, u_k$ is
  irreducible, $l = k$. Therefore,
  \begin{gather*}
    u_{l'} \in R_\Lambda(w) \cap \vartheta_\Lambda(\psi) \cap W_{l'} \subseteq V^b_k(w, \psi) =
    \emptyset.
  \end{gather*}
  Contradiction.

  By~\cref{cor:K4}, an irreducible $\Psi^\zeta$-labeled path of length $C \defeq n(M^2 + M + 1)$,
  where $M \defeq n^n|\Psi^\zeta|^n$, does not exist in $\calM_\Lambda$. Therefore,
  $W_C = \emptyset$ and $|W| = |\bigcup_{l < C}W_l| < |\Psi^\zeta|^C$. So, $\calF \defeq (W, R)$ is
  a finite subframe of $\calF_\Lambda$. By the condition of the theorem, $\calF \vDash \Lambda$. At
  the same time $\calF \nvDash \zeta$. Thus, $\Lambda$ has the finite model property.
 \subsection{Pretransitive analogues of GL}
  \label{s:GL-fmp} To prove~\cref{t:GL-fmp}, we need some extra lemmas.
  \begin{lemma}
    \label{l:u-1} For an $n$-transitive logic $\Lambda$ and $k \geq 1$,
    \begin{gather*}
      \Lambda \vdash \sigma_k(p) \wedge \neg\Dmd^{+n} q \impl \sigma_k(p \wedge \neg\Dmd^{+n} q).
    \end{gather*}
  \end{lemma}
  \begin{proof}
    We proceed by induction on $k$. For $k = 1$, the statement is clear.

    Suppose that $k > 1$ and the statement holds for $k-1$. Notice that
    \begin{align*}
      \Lambda \vdash \Dmd\Dmd^{+n} q &\impl \bigvee_{k=2}^n \Dmd^k q \vee \Dmd^{n+1}q
      &&\text{by normality}\\
      &\impl \Dmd^{+n} q, &&\text{since $\Lambda$ is $n$-transitive}.
    \end{align*}
    Therefore, $\Lambda \vdash \neg\Dmd^{+n} q \impl \Box\neg\Dmd^{+n} q$ and
    \begin{align*}
      \Lambda \vdash \sigma_k(p) \wedge \neg\Dmd^{+n} q &\impl \Dmd(\sigma_{k-1}(p) \wedge
      \neg\Dmd^{+n} q)
      &&\text{by normality}\\
      &\impl \Dmd\sigma_{k-1}(p \wedge \neg\Dmd^{+n} q)
      &&\text{by the induction hypothesis}.
    \end{align*}
    It is also clear that
    $\Lambda \vdash \sigma_k(p) \wedge \neg\Dmd^{+n} q \impl p \wedge \neg\Dmd^{+n} q$. Thus, the
    statement holds for $k$.
  \end{proof}
  \begin{lemma}
    \label{l:GL-irrefl} Suppose that $\Lambda \supseteq \mathrm{GL}_{\Dmd^2\sigma_n}$ for some
    $n \geq 1$, $w \in W_\Lambda$, $\varphi \in \mathrm{Fm}$. Then all worlds in
    $\mathrm{max}(\vartheta_\Lambda(\varphi))$ and
    $\mathrm{max}(R_\Lambda(w) \cap \vartheta_\Lambda(\varphi))$ are irreflexive.
  \end{lemma}
  \begin{proof}
    Suppose that $v \in \mathrm{max}(\vartheta_\Lambda(\varphi))$ is reflexive. Then
    $\calM_\Lambda, v \vDash \Dmd\varphi$ and, by \cref{l:GL-trans-lob},
    $\calM_\Lambda, v \vDash \Dmd(\varphi \wedge \neg\Dmd^{+n}\sigma_n(\varphi))$. Therefore,
    $\calM_\Lambda, u \vDash \varphi \wedge \neg\Dmd^{+n}\sigma_n(\varphi)$ for some
    $u \in R_\Lambda(v)$. Since $v$ is maximal, $u \mathrel R_\Lambda^+ v$, but this is impossible,
    because $\calM_\Lambda, v \vDash \sigma_n(\varphi)$.

    Now suppose that $v \in \mathrm{max}(R_\Lambda(w) \cap \vartheta_\Lambda(\varphi))$ is
    reflexive. Then $\calM_\Lambda, v \vDash \Dmd\sigma_n(\varphi)$. Notice that
    \begin{align*}
      \Lambda \vdash \Dmd\sigma_n(\varphi) &\impl \Dmd\bigl(\sigma_n(\varphi) \wedge
      \neg\Dmd^{+n}\sigma_n(\sigma_n(\varphi))\bigr)
      &&\text{by \cref{l:GL-trans-lob}}\\
      &\impl \Dmd\sigma_n\bigl(\varphi \wedge \neg\Dmd^{+n}\sigma_n(\sigma_n(\varphi))\bigr)
      &&\text{by \cref{l:u-1}}.
    \end{align*}
    Therefore, there are $u_1, \dots, u_n \in \vartheta_\Lambda\bigl(\varphi \wedge
    \neg\Dmd^{+n}\sigma_n(\sigma_n(\varphi))\bigr)$ such that
    $v \mathrel R_\Lambda u_1 \mathrel R_\Lambda \dots \mathrel R_\Lambda u_n$. Since
    $\calF_\Lambda \vDash {\rm A}4_{\Dmd^2\sigma_n}$, $w \mathrel R_\Lambda u_i$ for some
    $i \in \{1, \dots, n\}$. Notice that $u_i \in R_\Lambda(w) \cap \vartheta_\Lambda(\varphi)$ and
    $v \mathrel R^+_\Lambda u_i$. Since $v$ is maximal, $u_i \mathrel R^+_\Lambda v$. This is
    impossible, since $\calM_\Lambda, v \vDash \sigma_n(\sigma_n(\varphi))$.
  \end{proof}
  \begin{lemma}
    \label{l:GL-cycle} Suppose that $\Lambda \supseteq \mathrm{GL}_{\Dmd^2\sigma_n}$ for some
    $n \geq 1$. Let $u_0, \psi_0, \dots, u_m$ be an optimal labeled path in $\calM_\Lambda$ in which
    every world is irreflexive. Then $u_0 \notin R_\Lambda(u_m)$.
  \end{lemma}
  \begin{proof}
    The proof is by induction on $m$. Since $u_0$ is irreflexive, the base case $m = 0$ holds.
    Suppose that $m \geq 1$, $u_0, \psi_0, \dots, u_m$ is an optimal labeled path,
    $u_m \mathrel R_\Lambda u_0$, and the statement is true for all paths which length is less than
    $m$. In particular, the statement holds for the paths $u_k, \psi_k \dots u_{m-1}$, $k < m$, that
    is, $u_k \notin R_\Lambda(u_{m-1})$ for $k < m$. By tightness, there are
    $\eta_0, \dots, \eta_{m-1}$ such that
    \begin{gather*}
      \calM_\Lambda, u_k \vDash \eta_k \quad\text{and}\quad \calM_\Lambda, u_{m - 1} \nvDash
      \Dmd\eta_k.
    \end{gather*}
    Consider the formula $\varphi \defeq \bigvee_{k=0}^{m-1} \eta_k \vee \psi_{m-1}$. Since
    $\calM_\Lambda, u_{m-1} \vDash \Dmd\varphi$, by~\cref{l:GL-trans-lob}
    \begin{gather*}
      \calM_\Lambda, u_{m-1} \vDash \Dmd\bigl(\varphi \wedge \neg\Dmd^{+n}\sigma_n(\varphi)\bigr).
    \end{gather*}
    Moreover, since $\calM_\Lambda, u_{m-1} \nvDash \Dmd\eta_k$ for $k < m$,
    \begin{gather*}
      \calM_\Lambda, u_{m-1} \vDash \Dmd\bigl(\psi_{m-1} \wedge
      \neg\Dmd^{+n}\sigma_n(\varphi)\bigr).
    \end{gather*}
    Therefore, $\calM_\Lambda, v \vDash \psi_{m-1} \wedge \neg\Dmd^{+n}\sigma_n(\varphi)$ for some
    $v \in R_\Lambda(u_{m-1})$. Notice that $u_m \mathrel R_\Lambda^+ u_{m-1} \mathrel R_\Lambda v$.
    Since $u_m \in \mathrm{max}(R_\Lambda(u_{m-1}) \cap \vartheta_\Lambda(\psi_{m-1}))$,
    $v \mathrel R_\Lambda^+ u_m$. At the same time, $\varphi$ is true in all $u_k$, $k \leq m$, from
    which it is easy to show that $\sigma_n(\varphi)$ is also true in $u_k$. Contradiction.
  \end{proof}
  \begin{corollary}
    \label{cor:GL-C} Let $\Lambda$ be a logic extending $\mathrm{GL}_{\Dmd^2 \sigma_n}$, $\Psi$ be a
    finite set of formulas, $M \defeq n^n|\Psi|^n$. Then there is no optimal $\Psi$-labeled path in
    $\calM_\Lambda$ of length $C \defeq nM + n$ in which every world is irreflexive.
  \end{corollary}
  \begin{proof}
    Suppose that $w_0, \psi_0, \dots, w_C$ is an optimal $\Psi$-labeled path in $\calM_\Lambda$ in
    which all worlds are irreflexive. By \cref{cor:seq-out}, there is $k \leq nM$ such that
    $w_{k+n} \mathrel R^+_\Lambda w_k$. Then, by \cref{cor:4-cond-global},
    $w_{k+n} \mathrel R_\Lambda w_{k'}$ for some $k'$, $k \leq k' < k + n$. Thus, the labeled path
    $w_{k'}, \psi_{k'+1}, \dots, w_{k+n}$ contradicts \cref{l:GL-cycle}.
  \end{proof}
  Now, let us prove~\cref{t:GL-fmp}. Suppose that $\Lambda$ is valid in all finite conversely
  well-founded semisubframes of $\calF_\Lambda$ and contains $\mathrm{GL}_{\Dmd^2\sigma_n}$ for some
  $n \geq 1$. Let $\zeta$ be a formula non-derivable in $\Lambda$. Then $\{\neg\zeta\}$ is
  $\Lambda$-consistent and, by~\cref{l:max}, there is a world
  $x \in \mathrm{max}(\vartheta_\Lambda(\neg\zeta))$. We construct sets $W_k \subseteq W_\Lambda$
  and relations $S_k \subseteq R_\Lambda$ for $k \in \omega$ by recursion. Let us put
  $W_0 \defeq \{x\}$, $S_0 \defeq \emptyset$.

  Suppose that, for some $k \in \omega$, $W_k$ is defined. Consider the set
  \begin{gather*}
    D_k \defeq \{(w, \psi) \in W_k \times \Psi^\zeta \mid \calM_\Lambda, w \vDash \Dmd\psi \}.
  \end{gather*}
  For $(w, \psi) \in D_k$, we fix some elements
  $v_k(w, \psi) \in \mathrm{max}(R_\Lambda(w) \cap \vartheta_\Lambda(\psi))$. Now, we put
  \begin{gather*}
    W_{k+1} \defeq \{v_k(w, \psi) \mid (w, \psi) \in D_k\} \setminus \bigcup_{l \leq k}W_l,\\
    S_{k+1} \defeq \bigl\{\bigl(w, v_k(w, \psi)\bigr) \,\big|\, (w, \psi) \in D_k\bigr\}.
  \end{gather*}

  Finally, let $W \defeq \bigcup_{k \in \omega}W_k$, $S \defeq \bigcup_{k \in \omega}S_k$,
  $R \defeq R_\Lambda \cap S^*$, and $\vartheta \defeq \vartheta_\Lambda|_W$. Notice that, for all
  $w, v \in W$,
  \begin{align*}
    w \mathrel R^* v \wedge w \mathrel R_\Lambda v &\Rightarrow w \mathrel S^* v \wedge w
    \mathrel R_\Lambda v\\
    &\Rightarrow w \mathrel R v.
  \end{align*}
  Therefore, $\calF \defeq (W, R)$ is a semisubframe of $\calF_\Lambda$. Suppose that
  $R_\Lambda(w) \cap \vartheta_\Lambda(\psi) \neq \emptyset$ for some $w \in W$ and
  $\psi \in \Psi^\zeta$. Let $k \in \omega$ be such that $w \in W_k$. Then $(w, \psi) \in D_k$ and,
  by construction, there is $v = v_k(w, \psi) \in R_\Lambda(w) \cap \vartheta_\Lambda(\psi)$ such
  that $w \mathrel S v$, whence $w \mathrel R v$. Therefore, $\calM \defeq (\calF, \vartheta)$ is a
  $\zeta$-selective semisubmodel of $\calM_\Lambda$ and $\calM, x \nvDash \zeta$
  by~\cref{l:filtration}.

  Similarly to the proof of~\cref{t:K4-fmp}, one can check that if $v \in W_k$, then there is an
  optimal $\Psi^\zeta$-labeled path $u_0, \psi_0, \dots, u_k$ of length $k$ in $\calM_\Lambda$ such
  that $u_k = v$ and $u_l \in W_l$ for $l < k$. By~\cref{l:GL-irrefl}, $R_\Lambda|_W$ is
  irreflexive, whence, by~\cref{cor:GL-C}, there is no $\Psi^\zeta$-optimal labeled path of length
  $C \defeq n^{n+1}|\Psi^\zeta|^n + n$ in $\calM_\Lambda$. Therefore, $W_C = \emptyset$ and
  $|W| \leq |\bigcup_{l < C}W_l| < |\Psi^\zeta|^C$. So, $\calF = (W, R)$ is a finite semisubframe of
  $\calF_\Lambda$.

  Suppose that $\calF$ is not conversely well-founded. Then there is $w \in W$ such that
  $w \mathrel R^+ w$. Since $R_\Lambda|_W$ is irreflexive, $R = R_\Lambda \cap S^+$, whence
  $R^+ \subseteq S^+$. Therefore,
  \begin{gather*}
    w = u_0 \mathrel S u_1 \mathrel S \dots \mathrel S u_m \mathrel S w
  \end{gather*}
  for some $m \geq 1$, $u_1, \dots, u_m \in W$. By construction, there are formulas
  $\psi_0, \dots, \psi_{m-1}$ such that $u_0, \psi_0, \dots, u_m$ is an optimal labeled path. Since
  $u_m \mathrel S w$, $u_m \mathrel R_\Lambda w$, contradicting~\cref{l:GL-cycle}.

  So, $\calF$ is a finite conversely well-founded semisubframe of $\calF_\Lambda$. By the condition
  of the theorem, $\calF \vDash \Lambda$. At the same time, $\calF \nvDash \zeta$. Thus, $\Lambda$
  has the finite model property.

  \section{Pretransitive analogues of GL and L\"ob's rule}
\label{s:gl-equiv}
It is well-known~\cite[p.~59]{Bool93} that
\begin{gather*}
  \mathrm{GL} = \mathrm{K4} + \frac{\Box\varphi \impl \varphi}{\varphi} = \mathrm{wK4} +
  \frac{\Box\varphi \impl \varphi}{\varphi},
\end{gather*}
where $\Lambda + {R}$ means the closure of the logic $\Lambda$ under the rule $R$. In this section
we show that the same takes place for $\mathrm{GL}_{\Dmd\beta}$ and $\mathrm{(w)K4}_{\Dmd\beta}$.

\begin{proposition}
  \label{prop:gl-equiv} For an $n$-transitive logic $\Lambda$, the following logics coinside:
  \begin{enumerate}
    \item $\Lambda_1 \defeq \Lambda + \Dmd^{+n} p \impl \Dmd^{+n}(p \wedge \neg\Dmd^{+n} p)$;
    \item $\Lambda_2 \defeq \Lambda + \Dmd p \impl \Dmd^{+n}(p \wedge \neg\Dmd p)$;
    \item $\Lambda_3 \defeq \Lambda + \frac{\Box^{+n}\varphi \impl \varphi}{\varphi}$;
    \item $\Lambda_4 \defeq \Lambda + \frac{\Box \varphi \impl \varphi}{\varphi}$.
  \end{enumerate}
\end{proposition}
\begin{proof}
  It is clear that $\Lambda_1 \supseteq \Lambda_2$ and $\Lambda_3 \supseteq \Lambda_4$.

  Let us prove that $\Lambda_2 \supseteq \Lambda_3$. Suppose that
  $\Lambda_2 \vdash \Box^{+n}\varphi \impl \varphi$. Then
  $\Lambda_2 \vdash \Box\Box^{+n}\varphi \impl \Box^{+n}\varphi$. Therefore,
  $\Lambda_2 \vdash \Box^{+n}(\Box\Box^{+n}\varphi \impl \Box^{+n}\varphi)$. Since
  $\Lambda_2 \vdash \Box^{+n}(\Box p \impl p) \impl \Box p$,
  $\Lambda_2 \vdash \Box\Box^{+n}\varphi$. Then $\Lambda_2 \vdash \Box^{+n}\varphi$ and
  $\Lambda_2 \vdash \varphi$.

  Now we prove that $\Lambda_4 \supseteq \Lambda_1$. Notice that
  \begin{gather*}
    \mathrm{K} \vdash \Dmd^{+n} p \impl \Dmd\Dmd^{+n} p \vee \Dmd(p \wedge \neg\Dmd^{+n} p).
  \end{gather*}
  Also we have
  \begin{align*}
    \Lambda \vdash \Box\bigl(\Dmd^{+n} p \impl \Dmd^{+n}(p \wedge \neg\Dmd^{+n} p)\bigr) \wedge
    \Dmd\Dmd^{+n} p &\impl \Dmd\Dmd^{+n}(p \wedge \neg\Dmd^{+n} p)\\
    &\impl \Dmd^{+n}(p \wedge \neg\Dmd^{+n} p).
  \end{align*}
  Therefore,
  \begin{gather*}
    \Lambda \vdash \Box\bigl(\Dmd^{+n} p \impl \Dmd^{+n}(p \wedge \neg\Dmd^{+n} p)\bigr) \impl
    \bigl(\Dmd^{+n} p \impl \Dmd^{+n}(p \wedge \neg\Dmd^{+n} p)\bigr).
  \end{gather*}
  Thus, $\Lambda_4 \vdash \Dmd^{+n} p \impl \Dmd^{+n}(p \wedge \neg\Dmd^{+n} p)$.
\end{proof}
Denote ${\rm AL\ddot ob}^{+n} \defeq \Dmd^{+n} p \impl \Dmd^{+n}(p \wedge \neg\Dmd^{+n} p)$.
\begin{lemma}
  \label{l:trans-lob-frames} For an $n$-transitive frame $\calF$,
  $\calF \vDash {\rm AL\ddot ob}^{+n}$ iff $\calF$ is conversely well-founded.
\end{lemma}
\begin{proof}
  Since $\Dmd^{+n} p$ expresses transitive closure in $\calF = (W, R)$,
  \begin{align*}
    \calF \vDash {\rm AL\ddot ob}^{+n}
    &\Iff (W, R^+) \vDash \mathrm{GL},\\
    &\Iff R^+ \text{ is transitive and conversely well-founded}.
  \end{align*}
  $R^+$ is always transitive. It is conversely well-founded iff $R$ is conversely well-founded.
\end{proof}
\begin{corollary}
  \label{cor:55} If $n \geq d(\beta)$, then
  $\mathrm{K4}_{\Dmd\beta} + {\rm AL\ddot ob}^{+n} \subseteq \mathrm{GL}_{\Dmd\beta}$.
\end{corollary}
\begin{proof}
  By \cref{cor:fmp}, $\mathrm{GL}_{\Dmd\beta}$ is Kripke-complete. Let $\calF$ be a
  $\mathrm{GL}_{\Dmd\beta}$-frame. Then, by \cref{pr:GL-frames},
  $\calF \vDash \mathrm{K4}_{\Dmd\beta}$ and $\calF$ is well-founded. By~\cref{cor:inclusions}
  $\calF$ is $n$-transitive, therefore $\calF \vDash {\rm AL\ddot ob}^{+n}$ by
  \cref{l:trans-lob-frames}.
\end{proof}
\begin{lemma}
  \label{l:54} For all $n \geq 1$,
  $\mathrm{K4}_{\Dmd^2\sigma_n} + {\rm AL\ddot ob}^{+n} \supseteq \mathrm{GL}_{\Dmd^2\sigma_n}$.
\end{lemma}
\begin{proof}
  Notice that
  \begin{align*}
    \mathrm{K4}_{\Dmd^2\sigma_n} + {\rm AL\ddot ob}^{+n} \vdash \Dmd^{+n}\sigma_n &\impl
    \Dmd^{+n}(\sigma_n \wedge \neg\Dmd^{+n}\sigma_n)
    &&\text{by ${\rm AL\ddot ob}^{+n}$}\\
    &\impl \Dmd^{+n}\sigma_n(p \wedge \neg\Dmd^{+n}\sigma_n)
    &&\text{by \cref{l:u-1}}\\
    &\impl \Dmd(p \wedge \neg\Dmd^{+n}\sigma_n)
    &&\text{by \cref{l:GL-trans-lob}}.
  \end{align*}
  Therefore,
  \begin{align*}
    \mathrm{K4}_{\Dmd^2\sigma_n} + {\rm AL\ddot ob}^{+n} \vdash \Dmd p &\impl \Dmd(p \wedge
    \neg\Dmd\sigma_n) \vee \Dmd^2\sigma_n\\
    &\impl \Dmd(p \wedge \neg\Dmd\sigma_n).
  \end{align*}
\end{proof}
\begin{theorem}
  \label{t:gl-equiv} For $\Lambda \in \{\mathrm{K4}_{\Dmd\beta}, \mathrm{wK4}_{\Dmd\beta}\}$ the
  logics $\Lambda_1$, $\Lambda_2$, $\Lambda_3$, and $\Lambda_4$ from \cref{prop:gl-equiv} are equal
  to $\mathrm{GL}_{\Dmd\beta}$.
\end{theorem}
\begin{proof}
  Let $n = d(\beta) + 1$. Notice that $\Lambda_1 = \Lambda + {\rm AL\ddot ob}^{+n}$. By
  \cref{cor:55},
  \begin{gather*}
    \mathrm{wK4}_{\Dmd\beta} + {\rm AL\ddot ob}^{+n} \subseteq \mathrm{K4}_{\Dmd\beta} +
    {\rm AL\ddot ob}^{+n} \subseteq \mathrm{GL}_{\Dmd\beta}.
  \end{gather*}
  We will show that
  $L \defeq \mathrm{wK4}_{\Dmd\beta} + {\rm AL\ddot ob}^{+n} \supseteq \mathrm{GL}_{\Dmd\beta}$.
  Then, for $\Lambda \in \{\mathrm{wK4}_{\Dmd\beta}, \mathrm{K4}_{\Dmd\beta}\}$,
  $\Lambda_1 = \mathrm{GL}_{\Dmd\beta}$ and, by \cref{prop:gl-equiv},
  $\Lambda_2 = \Lambda_3 = \Lambda_4 = \mathrm{GL}_{\Dmd\beta}$.

  Notice that $\mathrm{wK4}_{\Dmd\beta} = \mathrm{K4}_{\Dmd^2\sigma_n} + {\rm Aw}4_{\Dmd\beta}$ by
  \cref{cor:inclusions} and
  $\mathrm{K4}_{\Dmd^2\sigma_n} + {\rm AL\ddot ob}^{+n} = \mathrm{GL}_{\Dmd^2\sigma_n}$ by
  \cref{cor:55,l:54}. Therefore
  \begin{gather*}
    L = \mathrm{wK4}_{\Dmd\beta} + {\rm AL\ddot ob}^{+n} = \mathrm{K4}_{\Dmd^2\sigma_n} +
    {\rm Aw}4_{\Dmd\beta} + {\rm AL\ddot ob}^{+n}
    = \mathrm{GL}_{\Dmd^2\sigma_n} + {\rm Aw}4_{\Dmd\beta}.
  \end{gather*}
  Since the formula ${\rm Aw}4_{\Dmd\beta}$ is canonical and, by \cref{pr:t-inherit},
  semisubframe-hereditary, $L$ is Kripke-complete by \cref{t:GL-fmp}. Let $\calF$ be an $L$-frame.
  Then $\calF \vDash \mathrm{wK4}_{\Dmd\beta}$ and, by \cref{l:trans-lob-frames}, $\calF$ is
  conversely well-founded. By \cref{pr:GL-frames}, $\calF \vDash \mathrm{GL}_{\Dmd\beta}$. Thus,
  $\mathrm{GL}_{\Dmd\beta} \subseteq L$.
\end{proof}
 \section{Some open problems}
  \label{s:problems}
  \subsection{Complexity}
    It is well-known that all logics between $\mathrm{K}$ and $\mathrm{S4}$ and between $\mathrm{K}$
    and $\mathrm{GL}$ are PSpace-hard~\cite{Lad77,Spaan93}. Therefore, $\mathrm{wK4}_\gamma$,
    $\mathrm{K4}_\gamma$, and $\mathrm{GL}_{\Dmd\beta}$ are PSpace-hard. The PSpace upper bound is
    known for $\mathrm{K4}$~\cite{Lad77} and $\mathrm{wK4}$~\cite{Shap05,Shap22}. For
    $\mathrm{wK4}^1_n$, $n \geq 2$ it was stated as an open problem in~\cite{KudShap25}. The method
    we used in this paper for establishing finite model property can be also used to prove PSpace
    decidability, as it was shown in~\cite{Shap05} for the transitive case. However, for our logics,
    filtration process should be modified to obtain this result. Therefore, we postpone it to
    another paper.
  \subsection{Subframe-hereditary pretransitive analogues of GL}
    In the proof of \cref{t:GL-fmp}, we extracted a finite conversely well-founded semisubmodel from
    the canonical model. It is unclear if one can always extract a submodel instead. To this end, it
    is interesting whether we can strengthen~\cref{l:GL-irrefl} in the following way: all clusters
    of maximal points in the canonical model of $\Lambda \supseteq \mathrm{GL}_{\Dmd^2\sigma_n}$ are
    irreflexive singletons. If this were the case, it would be possible to extract a submodel like
    in the proof of~\cref{t:K4-fmp}.
  \subsection{Weaker subframe-hereditary pretransitive logics}
    In \cref{ss:main} we noticed that the logics $\mathrm{K} + \mathrm{Trans}_n^s$ are
    pretransitive, canonical, and subframe-hereditary. In fact $\mathrm{K} + \mathrm{Trans}_n^s$ is
    the least $n$-transitive subframe-hereditary logic. However, these logics are weaker than
    $\mathrm{K4}_{\Dmd^2\sigma_n}$, whence it is still unknown if they have the finite model
    property.

  \bibliographystyle{alpha} \bibliography{defs/bib-new.bib}

\begin{thebibliography}{BdRV01}

\bibitem[BdRV01]{BdRV01}
Patrick Blackburn, Maarten de~Rijke, and Yde Venema.
\newblock {\em Modal logic}, volume~53 of {\em Camb. Tracts Theor. Comput.
  Sci.}
\newblock Cambridge: Cambridge University Press, 2001.

\bibitem[BEG11]{BEG09}
Guram Bezhanishvili, Leo Esakia, and David Gabelaia.
\newblock Spectral and {{\({T} _{0}\)}}-spaces in d-semantics.
\newblock In {\em Logic, language, and computation. 8th international Tbilisi
  symposium on logic, language, and computation, TbiLLC 2009, Bakuriani,
  Georgia, September 21--25, 2009. Revised selected papers}, pages 16--29.
  Berlin: Springer, 2011.

\bibitem[BGJ11]{BGJ11}
Guram Bezhanishvili, Silvio Ghilardi, and Mamuka Jibladze.
\newblock An algebraic approach to subframe logics. {Modal} case.
\newblock {\em Notre Dame J. Formal Logic}, 52(2):187--202, 2011.

\bibitem[Boo93]{Bool93}
George Boolos.
\newblock {\em The logic of provability}.
\newblock Cambridge: Cambridge University Press, 1993.

\bibitem[CZ97]{ChZa97}
Alexander Chagrov and Michael Zakharyaschev.
\newblock {\em Modal logic}, volume~35 of {\em Oxf. Logic Guides}.
\newblock Oxford: Clarendon Press, 1997.

\bibitem[Dvo24]{Dvo24}
Lev~V. Dvorkin.
\newblock On provability logics of {Niebergall} arithmetic.
\newblock {\em Izv. Math.}, 88(3):468--505, 2024.

\bibitem[Fin85]{Fine85}
Kit Fine.
\newblock Logics containing {K4}. {II}.
\newblock {\em J. Symb. Log.}, 50:619--651, 1985.

\bibitem[Gab72]{Gab72}
Dov~M. Gabbay.
\newblock A general filtration method for modal logics.
\newblock {\em J. Philos. Log.}, 1:29--34, 1972.

\bibitem[GSS09]{GabSkvoSheht09}
Dov~M. Gabbay, Dimitrij Skvortsov, and Valentin Shehtman.
\newblock {\em Quantification in nonclassical logic. {Volume} {I}.}, volume 153
  of {\em Stud. Logic Found. Math.}
\newblock Amsterdam: Elsevier, 2009.

\bibitem[Jan94]{Jan94}
Ramon Jansana.
\newblock Some logics related to von {Wright}'s logic of place.
\newblock {\em Notre Dame J. Formal Logic}, 35(1):88--98, 1994.

\bibitem[KK06]{KowKra06}
Tomasz Kowalski and Marcus Kracht.
\newblock Semisimple varieties of modal algebras.
\newblock {\em Stud. Log.}, 83(1-3):351--363, 2006.

\bibitem[Kra99]{Kra99}
M.~Kracht.
\newblock {\em Tools and techniques in modal logic}, volume 142 of {\em Stud.
  Logic Found. Math.}
\newblock Amsterdam: Elsevier, 1999.

\bibitem[KS11]{KudShap11}
Andrey Kudinov and Ilya Shapirovsky.
\newblock Finite model property of pretransitive analogs of {S5}.
\newblock {\em Topology, Algebra, and Categories in Logic}, pages 261--264,
  2011.

\bibitem[KS17]{KudShap17}
Andrey~V. Kudinov and Ilya~B. Shapirovsky.
\newblock Partitioning {Kripke} frames of finite height.
\newblock {\em Izv. Math.}, 81(3):592--617, 2017.

\bibitem[KS25]{KudShap25}
Andrey Kudinov and Ilya Shapirovsky.
\newblock Two types of filtrations for {wK4} and its relatives.
\newblock Preprint, {arXiv}:2401.00457 [math.{LO}] (2025), 2025.

\bibitem[Kur18]{Kur18}
Taishi Kurahashi.
\newblock Arithmetical soundness and completeness for {{\(\Sigma_{2}\)}}
  numerations.
\newblock {\em Stud. Log.}, 106(6):1181--1196, 2018.

\bibitem[Lad77]{Lad77}
Richard~E. Ladner.
\newblock The computational complexity of provability in systems of modal
  propositional logic.
\newblock {\em SIAM J. Comput.}, 6:467--480, 1977.

\bibitem[Lem66]{Lemm66}
E.~J. Lemmon.
\newblock Algebraic semantics for modal logics. {I}, {II}.
\newblock {\em J. Symb. Log.}, 31:46--65, 191--218, 1966.

\bibitem[Sac01]{Sacc01}
Lorenzo Sacchetti.
\newblock The fixed point property in modal logic.
\newblock {\em Notre Dame J. Formal Logic}, 42(2):65--86, 2001.

\bibitem[Sah75]{Sahlq75}
Henrik Sahlqvist.
\newblock Completeness and correspondence in the first and second order
  semantics for modal logic.
\newblock Proc. 3rd {Scand}. {Logic} {Symp}., {Uppsala} 1973, 110-143 (1975).,
  1975.

\bibitem[Sha05]{Shap05}
I.~Shapirovsky.
\newblock On {PSPACE}-decidability in transitive modal logics.
\newblock In {\em Advances in modal logic. Vol. 5. Selected papers from the 5th
  conference (AiML 2004), Manchester, UK, September 9--11, 2004}, pages
  269--287. London: King's College Publications, 2005.

\bibitem[Sha21]{Shap21}
Ilya~B. Shapirovsky.
\newblock Glivenko's theorem, finite height, and local tabularity.
\newblock {\em J. Appl. Log. - IfCoLog J. Log. Appl.}, 8(8):2333--2347, 2021.

\bibitem[Sha22]{Shap22}
Ilya Shapirovsky.
\newblock Satisfiability problems on sums of kripke frames.
\newblock {\em ACM Transactions on Computational Logic}, 23, 04 2022.

\bibitem[Spa93]{Spaan93}
Edith Spaan.
\newblock {\em Complexity of Modal Logics}.
\newblock PhD thesis, University of Amsterdam, 1993.

\bibitem[SS03]{ShapSheht03}
Ilya Shapirovsky and Valentin Shehtman.
\newblock Chronological future modality in {Minkowski} spacetime.
\newblock In {\em Advances in modal logic. Vol. 4. Selected papers from the 4th
  conference (AiML 2002), Toulouse, France, October 2002}, pages 437--459.
  London: King's College Publications, 2003.

\bibitem[Wol93]{Wolt93}
Frank Wolter.
\newblock {\em Lattices of modal logics}.
\newblock PhD thesis, Freie Universit{\"a}t Berlin, 1993.

\end{thebibliography}
\end{document}